\pdfoutput=1 
\documentclass[final]{cmslatex}
\usepackage{graphicx}
\usepackage[update]{epstopdf} 
\usepackage{url}
\usepackage{hyperref}
\usepackage{graphicx} 
\usepackage{epstopdf}

\usepackage{bm}
\usepackage{amsmath}
\usepackage{amssymb}
\usepackage{graphicx}
\usepackage{subcaption} 
\usepackage{color}

\usepackage{amsfonts}
\usepackage{amscd}
\usepackage{mathrsfs}
\usepackage{enumerate}
\usepackage{url}
\providecommand{\email}[1]{\url{#1}} 
\usepackage{algorithm}
\usepackage{algpseudocode}
\usepackage{amsmath, amssymb, bm, mathtools}
\usepackage{empheq}
\usepackage{tikz}
\usetikzlibrary{arrows.meta, decorations.pathreplacing}

\newenvironment{equationa*}{\begin{equation*}\begin{aligned}} {\end{aligned}\end{equation*}}
\allowdisplaybreaks[2]

\begin{document}

\title{High-Order Quadrature for Implicitly Defined Surfaces and Regions}

\author{
Zibo Zhao \thanks{Qiu Zhen College, Tsinghua University, Beijing, China (\texttt{zhaozb21@mails.tsinghua.edu.cn}).}
\and
Zuoqiang Shi \thanks{Yau Mathematical Sciences Center, Tsinghua University,
Beijing, China, 100084. \&
Yanqi Lake Beijing Institute of Mathematical Sciences and Applications,
 Beijing, China, 101408.
\email{zqshi@tsinghua.edu.cn}}}

\maketitle

\begin{abstract}
We propose a high-order quadrature method for integrating over curves, surfaces, and enclosed regions defined implicitly as zero level sets of smooth functions. The method combines a mesh adjustment procedure with local change-of-variables parametrizations on cut elements. The mesh adjustment moves background mesh vertices away from the interface, ensuring a simple and nondegenerate intersection pattern between the level set and the simplicial mesh. On each cut element, the interface is parametrized by solving one-dimensional nonlinear equations along prescribed rays, reducing surface and region integration to standard Gauss--Legendre quadrature on fixed reference domains. The resulting curve and surface quadrature rules are local, nonrecursive, and have strictly positive weights. We prove high-order accuracy for curve, surface, and region integrals in two and three dimensions under natural smoothness and mesh consistency assumptions. Numerical experiments on representative implicit geometries confirm the predicted convergence rates.
\end{abstract}

\section{Introduction}
We develop a high-order quadrature method for integrals over hypersurfaces given as the zero level set of a smooth function \(F: \mathbb{R}^d \rightarrow \mathbb{R}\). Let \(\Gamma=\{x:F(x)=0\}\). Throughout the paper, \(\Omega\) denotes the bounded region enclosed by \(\Gamma\), and the sign of \(F\) is chosen so that \(\Omega=\{x:F(x)\le 0\}\). Equivalently, one may work with \(\Omega=\{F\le0\}\cap U\) in a fixed computational box \(U\), provided that every boundary component created by \(\partial U\) is treated separately. We assume that \(\overline\Omega\) is compactly contained in the computational box used by the background mesh. We consider
\[
\int_{\Gamma} f\,dS,\qquad
\int_{\Omega} f\,dx,
\]
We assume that \(\Gamma\) is a compact regular level set, so that
\(\nabla F\neq 0\) on \(\Gamma\). For a fixed \(N_q\)-point quadrature rule, we
assume that \(F\) and \(f\) have sufficiently many bounded derivatives in a
neighborhood of the integration domain; for example,
\(F\in C^{2N_q+2}\) and \(f\in C^{2N_q}\) are sufficient for the estimates
used below. Constants in the analysis may depend on these smoothness bounds,
the geometry of \(\Gamma\), the mesh shape-regularity constant, and the
uniform transversality constant, but not on the mesh size \(h\).

Implicitly defined surfaces and domains arise naturally in level set methods, interface tracking, surface finite element methods, unfitted finite element methods, and finite-cell or cut finite element discretizations \cite{Osher1988,Dziuk2013,Burman2015,abedian2013performance}. In these settings, accurate integration over surfaces, cut cells, or implicitly bounded regions is essential. Since the geometry is given only through a level set function, the main difficulty is to construct quadrature rules that capture the curved geometry without reducing the accuracy of the underlying high-order discretization.

Several classes of methods have been developed. One approach explicitly resolves the geometry before applying standard quadrature. Min and Gibou \cite{min2007geometric} decompose irregular domains into simplices and reconstruct the interface by piecewise linear interpolation. Related surface extraction techniques include marching cubes and marching tetrahedra \cite{Lorensen1987,Payne1990,Gueziec1995}. Higher-order reconstruction-based schemes, such as the method of Fries and Omerovi\'c \cite{Thomas2020}, improve geometric accuracy for implicit domains. These methods are intuitive and compatible with standard finite element assembly, but their accuracy depends on the reconstructed geometry, and small cuts, topology changes, and high-order curved interfaces can be delicate.

A second approach uses Dirac delta or Heaviside functions to rewrite surface and region integrals on a fixed background grid. Since these generalized functions are singular or discontinuous, regularization is required. Engquist, Tornberg, and Tsai \cite{engquist2005discretization} studied delta-function discretizations in level set methods, and related robust approximations were developed in \cite{Min2008,smereka2006numerical,towers2007two,zahedi2010delta}. Wen derived high-order discretizations for delta-function integrals in one, two, and three dimensions \cite{Wen2008,Wen2009,Wen2010}. Although such methods avoid explicit reconstruction, their accuracy depends on the regularization width, local resolution, and the quality of the level set representation.

Other methods rely on integral identities or algebraic construction of quadrature rules. The coarea formula \cite{Federer} has been used to approximate integrals over implicitly defined hypersurfaces \cite{drescher2017}. Moment-fitting methods construct weights by enforcing exactness conditions \cite{muller2013highly}. These methods can be highly accurate and efficient, but the resulting weights may be non-positive, which can be undesirable for stability or positivity preservation.

Saye \cite{saye2015} proposed a recursive high-order quadrature method on
hyperrectangles.  A coordinate direction is selected in which a component of
\(\nabla F\) remains nonzero, the interface is represented as the graph of a
height function, and the resulting base integral is treated recursively in one
fewer spatial dimension.  This construction is general and highly accurate,
but it may require subdivision when no single coordinate direction is valid on
the whole cell.  For example, on
\(\Gamma=\{x^2+y^2=1\}\), neither \(\partial_xF\) nor
\(\partial_yF\) remains uniformly nonzero on a box containing the entire
circle.

Saye later extended this dimension-reduction framework to domains described by
one or more multivariate polynomials \cite{saye2022}.  The resulting method
handles multi-component geometry, Boolean combinations, junctions, and
singular algebraic sets by constructing implicitly defined multivalued height
functions.  It also admits an extension from hyperrectangles to simplex
constraint cells by introducing additional polynomial inequalities.  For
smooth problems, the reported \(h\)-convergence rate is \(2q\) when the
underlying one-dimensional Gauss--Legendre rule has \(q\) nodes, and
\(q\)-refinement is approximately exponential.  The price for this generality
is a recursive geometric decomposition based on projected critical sets and
lower-dimensional implicit domains.

Cui, Leng, Liu, Zhang, and Zheng \cite{cui2020} developed a direct
high-order algorithm for a tetrahedron cut by an implicitly defined curved
interface.  Their construction decomposes the volume and surface integrals
into a sequence of one-dimensional integrals after selecting suitable
integration directions and locating the nonessential singularities of the
reduced integrands.  The method requires only univariate root finding, produces
positive weights with quadrature points in the integration domain, and has
been implemented in the PHG finite element toolbox.  Thus, one-dimensional
root finding and positivity by themselves are not the distinguishing features
of the present work.  Related graph- and direction-based constructions also
appear in \cite{beck2023,olshanskii2016numerical}.

The contribution of the present paper is complementary.  First, we give a
complete vertex-displacement algorithm that converts a shape-regular
background simplicial mesh into a mesh having a uniformly simple and
nondegenerate intersection with the level set.  The induced curve and surface
patches therefore form a geometrically well-conditioned local mesh on
\(\Gamma\), without reconstructing \(\Gamma\) by polynomial surface elements.
The proof shows that the adjusted background mesh remains shape-regular and
that every cut simplex satisfies a uniform separation and transversality
condition.  Second, this geometry permits a nonrecursive radial
parametrization on each local patch, for which we prove uniform high-order
regularity estimates and rigorous global quadrature-error bounds for curve,
surface, and enclosed-region integrals.  These two ingredients---a proved
mesh-generation procedure and a complete error analysis for the resulting
radial charts---are the main innovations of the paper.

\subsection{Overview of the proposed method.}
The proposed method consists of three steps. First, the computational box is partitioned into a simplicial background mesh. The mesh vertices near the zero level set are then slightly displaced away from the interface. This mesh adjustment does not approximate the surface; rather, it guarantees that each cut element has a simple and nondegenerate intersection pattern with $\Gamma$. In particular, no vertex of a cut element lies too close to the interface, and the radial directions used later are uniformly transverse to $\Gamma$.


Second, on each local chart, the interface is parametrized by solving a one-dimensional nonlinear equation along prescribed rays. For example, after translating the apex vertex to the origin, points on the interface are written as
\begin{equation}
z(\lambda)=\alpha(\lambda)x(\lambda),
\qquad
F(\alpha(\lambda)x(\lambda))=0.
\end{equation}
Region integrals are treated by adding one radial variable:
\begin{equation}
\phi(\lambda,\beta)=\beta z(\lambda),
\qquad
0\le \beta\le 1.
\end{equation}
The details are given in Section \ref{sec:2dcurve},\ref{sec:2dregion}, section \ref{sec:3dplane} and section \ref{sec:3dregion}. 

Finally, the transformed integrals are evaluated by standard Gauss--Legendre
quadrature rules on fixed reference intervals or triangles. For curve and surface
integrals, the reference weights and the absolute Jacobian factors are
positive. Region integrals are evaluated by standard Gauss--Legendre
quadrature rules on fixed reference triangles or tetrahedrons. The regularity estimates
in section \ref{sec:order} show that the transformed integrands have the required smoothness
and scaling properties, which leads to the global error estimate of order
$O(h^{N_q})$ for curve, surface, and region integrals.

The paper is organized as follows: Basic geometrical analysis is present in Section 2. Sections 3 and Section 4 gives the algorithm for 2D and 3D case respectively. We analyze the regularity of the parametrization and associated error analysis in Section 5. Section \ref{sec:surfacetest} gives numerical experiments. Section \ref{sec:conclusion} concludes the paper.

\section{Intersection between Mesh and Surface}
\label{sec:proof}

Let \(\Gamma = \{x \in \mathbb{R}^3: F(x) = 0\}\) be a smooth compact surface, with \(\kappa_1(x), \kappa_2(x)\) denoting the principal curvatures at \(x \in \Gamma\). Define:
\[
K_\Gamma
=
\max_{x\in\Gamma}\max\{|\kappa_1(x)|,|\kappa_2(x)|\}.
\]
Then we introduce the definition of reach for smooth manifold.
\begin{definition}
Let \(M \subset \mathbb{R}^3\) be a smooth submanifold. The
\textbf{reach} of \(M\), denoted by \(r_M\), is the supremum of all
\(r>0\) such that every point \(x\) with
\(\operatorname{dist}(x,M)<r\) has a unique closest point
\(\pi_M(x)\in M\).  Thus
\[
\operatorname{dist}(x,\pi_M(x))=\operatorname{dist}(x,M)
\qquad\text{whenever }\operatorname{dist}(x,M)<r_M.
\]
\end{definition}

It is well known that the reach is strictly positive for compact smooth manifold.  
\begin{proposition} \cite{Federer}
For a compact smooth submanifold \(M \subset \mathbb{R}^3\), the reach \(r_M > 0\).
\end{proposition}
Then we introduce an important geometrical result which will be used a lot through the paper. 
\begin{lemma}
\label{lem:h2}
 For any \(\bm{p,q} \in \mathbb{R}^3\) with \(h = \|\bm{p-q}\| < r_\Gamma\) and \(h K_\Gamma \leq 1/4\), if the line segment \(\overline{\bm{pq}}\) intersects \(\Gamma\) more than once, then:
\[
\max\{\mathrm{dist}(\bm{p}, \Gamma), \mathrm{dist}(\bm{q}, \Gamma)\} \leq K_\Gamma h^2.
\]
Here, intersections are counted with multiplicity. A tangential contact is counted as two intersections.
\end{lemma}

\begin{proof}
Let \(d(\cdot, \Gamma)\) be the signed distance function for \(\Gamma\). Set
\[
\bm v=\frac{\bm q-\bm p}{\|\bm q-\bm p\|},
\qquad
\bm x(t)=\bm p+t\bm v,
\qquad
0\le t\le h,
\]
and let \(d(t)=d(\bm x(t),\Gamma)\). We first show \(d(t)\in C^1[0,h]\).

Using the implicit function theorem, let \(\bm{y} = \pi_\Gamma(\bm x)\) be the closest point projection of \(\bm{x}\). The implicit relation \(G(t, \bm{y}, \xi) = 0\) is defined as:
\[
G(t, \bm{y}, \xi) = \begin{pmatrix}
F(\bm{y}) \\
\bm{x}(t) - \bm{y} - \xi \nabla F(\bm{y})
\end{pmatrix}.
\]
The Jacobian \(\frac{\partial G}{\partial (\bm{y}, \xi)}\) is invertible under the smoothness of \(F\) and small \(\xi\), ensuring \(d(t)\) is smooth.

Let \(\bm{n}(\bm{y})\) denote the outer normal of \(\Gamma\) at \(\bm{y}\). The projection satisfies:
\[
\bm{x}(t) = \pi(\bm{x}(t)) + d(t) \bm{n}(\pi(\bm{x}(t))).
\]
Differentiating with respect to \(t\) yields:
\[
\bm{v} = \frac{d\pi}{dt} + d'(t) \bm{n\circ \pi\circ x}(t) + d(t) \frac{d}{dt}\bm{n\circ \pi\circ x}(t).
\]
Since \(\frac{d\pi}{dt} \cdot \bm{n} = 0\) and \(\frac{d\bm{n}}{dt} \cdot \bm{n} = 0\), we have:
\[
d'(t) = \bm{v} \cdot \bm{n}(\pi(\bm{x}(t))), \quad \bm{v}_{\text{tan}} = (I - d(t)S) \frac{d\pi}{dt},
\]
where \(\bm{v}_{\text{tan}} = \bm{v} - (\bm{v} \cdot \bm{n}) \bm{n}\), and \(S\) is the Weingarten map with principal curvatures \(\kappa_1, \kappa_2\) and eigenvectors \(\bm{e}_1, \bm{e}_2\).

We calculate the second order derivative further.
\[
d''(t) = -\bm{v} \cdot S (I - d(t) S)^{-1} \bm{v}_{\text{tan}} = -\sum_{i=1}^2 \frac{\kappa_i v_i^2}{1 - d(t) \kappa_i},
\]
where \(v_i = \bm{v} \cdot \bm{e}_i\). Since \(|d(t)| \leq h\) and
\(|d(t)\kappa_i|\le hK_\Gamma\le 1/4\), we have
\[
|d''(t)|
\leq
\frac{4}{3}\sum_{i=1}^2|\kappa_i|v_i^2
\leq
\frac{4}{3}K_\Gamma
\leq
2K_\Gamma.
\]

If the two intersections occur at distinct parameters
\(t_1,t_2\in[0,h]\), then \(d(t_1)=d(t_2)=0\), and the standard
two-point interpolation remainder gives
\[
|d(t)| \leq \frac{1}{2} \left(\max_{t \in [0, h]} |d''(t)|\right) |(t - t_1)(t - t_2)|,
\]
If the segment has a tangential contact at \(t_1=t_2\), then
\(d(t_1)=d'(t_1)=0\); Taylor's theorem, or equivalently the Hermite
interpolation remainder, gives the same estimate with
\((t-t_1)^2\).  Consequently, in both cases
\[
\max\{\mathrm{dist}(\bm p, \Gamma), \mathrm{dist}(\bm q, \Gamma)\} = \max\{|d(0)|, |d(h)|\} \leq K_\Gamma h^2.
\]
\end{proof}

The intersection between tetrahedral mesh and surface can be very much simplified on the basis of Lemma \ref{lem:h2}. 

Let $T=\overline{\bm{o}_0\bm{o}_1\bm{o}_2\bm{o}_3}$ be a tetrahedral in $\mathbb{R}^3$ and $\bm{o}_i,\; i=0,1,2,3$ are vertices of $T$.  
\[
h_T = \max_{\bm{x},\bm{y}\in T} \|\bm{x}-\bm{y}\|=\max_{i,j\in \{0,1,2,3\}} \|\bm{o}_i-\bm{o}_j\|.
\]

\begin{definition}\label{def:consistent}
\(T=\overline{\bm{o}_0\bm{o}_1\bm{o}_2\bm{o}_3}\) is \textbf{consistent} with \(\Gamma\) as $h_T\rightarrow 0$  if:
\begin{itemize}
    \item \(h_T < r_\Gamma\), where \(r_\Gamma\) is the reach of \(\Gamma\).
     \item \(h_T K_\Gamma\le 1/4\)
    \item $\exists c_0>0$ independent on $h_T$, \(\text{\rm dist}(\bm{o}_i, \Gamma) > \max\{ c_0h_T, K_\Gamma h_T^2\},\; i=0,1,2,3\).
\end{itemize}
\end{definition}

The intersection between $T$ and surface $\Gamma$ can be summarized as following theorem. 
\begin{theorem}
\(\Gamma = \{x \in \mathbb{R}^3: F(x) = 0\}\) is a smooth compact surface.  $T=\overline{\bm{o}_0\bm{o}_1\bm{o}_2\bm{o}_3}$ is a tetrahedron consistent with \(\Gamma\).
\begin{itemize}
    \item[] Case 1: $$\text{\rm sgn}(F(\bm{o}_0))=\text{\rm sgn}(F(\bm{o}_1))=\text{\rm sgn}(F(\bm{o}_2))=\text{\rm sgn}(F(\bm{o}_3)).$$ 
    In this case $$\Gamma \cap T = \emptyset.$$
    
    \item[] Case 2: As shown in Fig. \ref{fig:cut-3D} 
    $$-\text{\rm sgn}(F(\bm{o}_0))=\text{\rm sgn}(F(\bm{o}_1))=\text{\rm sgn}(F(\bm{o}_2))=\text{\rm sgn}(F(\bm{o}_3))$$ 
    In this case, define the map 
    \begin{align*}
        \varphi:\quad &T\cap \Gamma \longrightarrow\;  \triangle \bm{o}_1\bm{o}_2\bm{o}_3\\
    & \bm{p}\;\longmapsto \;\overline{\bm{o}_0 \bm{p}}\cap \triangle \bm{o}_1\bm{o}_2\bm{o}_3. 
    \end{align*}
    $\varphi$ is bijective.
    
    \item[] Case 3: As shown in Fig. \ref{fig:cut-3D} 
    $$-\text{\rm sgn}(F(\bm{o}_0))=-\text{\rm sgn}(F(\bm{o}_1))=\text{\rm sgn}(F(\bm{o}_2))=\text{\rm sgn}(F(\bm{o}_3))$$  
    In this case, denote
    \[
    \bm p=\overline{\bm o_1\bm o_2}\cap\Gamma,\qquad
    T_1=\overline{\bm o_0\bm o_1\bm o_3\bm p},\qquad
    T_2=\overline{\bm o_0\bm o_2\bm o_3\bm p}.
    \]
    Define the two local projection maps
    \begin{align*}
        \varphi_1:\quad&T_1\cap\Gamma
        \longrightarrow \triangle\bm o_0\bm o_1\bm p,
        &
        \bm q&\longmapsto
        \overline{\bm o_3\bm q}\cap
        \triangle\bm o_0\bm o_1\bm p,\\
        \varphi_2:\quad&T_2\cap\Gamma
        \longrightarrow \triangle\bm o_2\bm o_3\bm p,
        &
        \bm q&\longmapsto
        \overline{\bm o_0\bm q}\cap
        \triangle\bm o_2\bm o_3\bm p.
    \end{align*}
    Both maps are bijections on their respective subpatches. Together the two
    local parametrizations cover \(T\cap\Gamma\). No global-bijection claim is
    made for the ordinary union of the two parameter triangles.
\end{itemize}
\label{thm:main}
\end{theorem}

\begin{proof}
Let $d$ denote the signed distance function to $\Gamma$, chosen with the same sign as $F$ in a tubular neighborhood of $\Gamma$. Since $\nabla F\neq 0$ on $\Gamma$, the sign of $F$ agrees with the sign of $d$ near $\Gamma$. 

We first give a useful consequence of the consistency condition. Let $S$ be any simplex contained in $T$, and suppose that all vertices of $S$ have the same sign. Then
\begin{equation}
S\cap\Gamma=\emptyset.
\label{eq:same-sign-simplex-empty}
\end{equation}
Indeed, assume for example that all vertices $a_i$ of $S$ satisfy $d(a_i)>0$. Since $T$ is consistent with $\Gamma$,
\begin{equation}
d(a_i)=\operatorname{dist}(a_i,\Gamma)>K_\Gamma h_T^2.
\end{equation}
Suppose, for contradiction, that there exists $x\in S\cap\Gamma$. Since $x\in S$, there are coefficients $\lambda_i\ge 0$ with $\sum_i\lambda_i=1$ such that
\begin{equation}
x=\sum_i\lambda_i a_i.
\end{equation}
Taylor expansion of the signed distance function at $x\in\Gamma$ and the estimate of $d''$ in Lemma \ref{lem:h2} give
\begin{equation}
d(a_i)
=
n(x)\cdot(a_i-x)+R_i,
\qquad
|R_i|\le K_\Gamma h_T^2,
\end{equation}
where $n(x)=\nabla d(x)$. Hence
\begin{equation}
n(x)\cdot(a_i-x)>0
\end{equation}
for every vertex $a_i$ of $S$. Therefore
\begin{equation}
0
=
n(x)\cdot(x-x)
=
\sum_i\lambda_i n(x)\cdot(a_i-x)
>
0,
\end{equation}
which is impossible. The case $d(a_i)<0$ for all vertices is identical. This proves \eqref{eq:same-sign-simplex-empty}.

We now prove the three cases.

\medskip

\noindent\textbf{Case 1.}
Assume that all four vertices of $T$ have the same sign. Applying \eqref{eq:same-sign-simplex-empty} with $S=T$, we obtain
\begin{equation}
T\cap\Gamma=\emptyset.
\end{equation}

\medskip

\noindent\textbf{Case 2.}
Assume, without loss of generality, that $o_0$ has one sign and $o_1,o_2,o_3$ have the opposite sign. Let
\begin{equation}
B=\operatorname{conv}{o_1,o_2,o_3}
\end{equation}
be the face opposite $o_0$.

Since the vertices of $B$ have the same sign, \eqref{eq:same-sign-simplex-empty} gives
\begin{equation}
B\cap\Gamma=\emptyset.
\end{equation}
Therefore $F$ has one fixed sign on $B$, opposite to the sign of $F(o_0)$.

For each $y\in B$, define the segment
\begin{equation}
\ell_y={(1-t)o_0+ty:0\le t\le 1}.
\end{equation}
The endpoints $o_0$ and $y$ have opposite signs, so by continuity $\ell_y$ intersects $\Gamma$ at least once.

We claim that this intersection is unique. If $\ell_y$ intersected $\Gamma$ more than once, counted with multiplicity, then Lemma~2.1 applied to the segment $\ell_y$ would imply
\begin{equation}
\max{\operatorname{dist}(o_0,\Gamma),\operatorname{dist}(y,\Gamma)}
\le
K_\Gamma |\ell_y|^2
\le
K_\Gamma h_T^2.
\end{equation}
In particular,
\begin{equation}
\operatorname{dist}(o_0,\Gamma)\le K_\Gamma h_T^2,
\end{equation}
which contradicts the consistency condition. Hence every segment $\ell_y$ intersects $\Gamma$ exactly once.

Define
\begin{equation}
\phi:T\cap\Gamma\longrightarrow B
\end{equation}
by sending $x\in T\cap\Gamma$ to the intersection point of the ray from $o_0$ through $x$ with the face $B$. Since each segment from $o_0$ to a point of $B$ intersects $\Gamma$ exactly once, the map $\phi$ is both injective and surjective. Therefore $\phi$ is bijective.

\medskip

 \noindent\textbf{Case 3.}
The edge \(\overline{o_1o_2}\) has exactly one intersection \(p\) with
\(\Gamma\); otherwise Lemma~\ref{lem:h2} would contradict the consistency
condition at \(o_1\) or \(o_2\). Set
\[
B_1=\operatorname{conv}\{o_0,o_1,p\},
\qquad
B_2=\operatorname{conv}\{o_2,o_3,p\}.
\]
The same segment argument used in Case 2 shows that
\[
B_1\cap\Gamma=\{p\},
\qquad
B_2\cap\Gamma=\{p\}.
\]
Indeed, an additional zero on \(B_1\), for example, would lie on a segment
whose endpoints have the same sign, producing at least two intersections
counted with multiplicity and contradicting Lemma~\ref{lem:h2}; the argument
for \(B_2\) is identical.

For every \(y\in B_1\setminus\{p\}\), the endpoints \(o_3\) and \(y\) have
opposite signs, while the segment \(\overline{o_3p}\) already meets
\(\Gamma\) at \(p\). Existence follows by continuity, and uniqueness follows
from Lemma~\ref{lem:h2} and the consistency condition at \(o_3\). Hence
\(\varphi_1:T_1\cap\Gamma\to B_1\) is bijective. The same argument, with apex
\(o_0\), proves that \(\varphi_2:T_2\cap\Gamma\to B_2\) is bijective.
Since \(T=T_1\cup T_2\), the two parametrizations cover \(T\cap\Gamma\).
Their subpatches may share only the curve contained in the common face
\(T_1\cap T_2\), which has zero surface measure, so the two local surface
integrals may be added without a global-bijection claim.
\end{proof}

\begin{figure}[htbp]
        \begin{tabular}{cc}
         \includegraphics[width=0.45\textwidth]{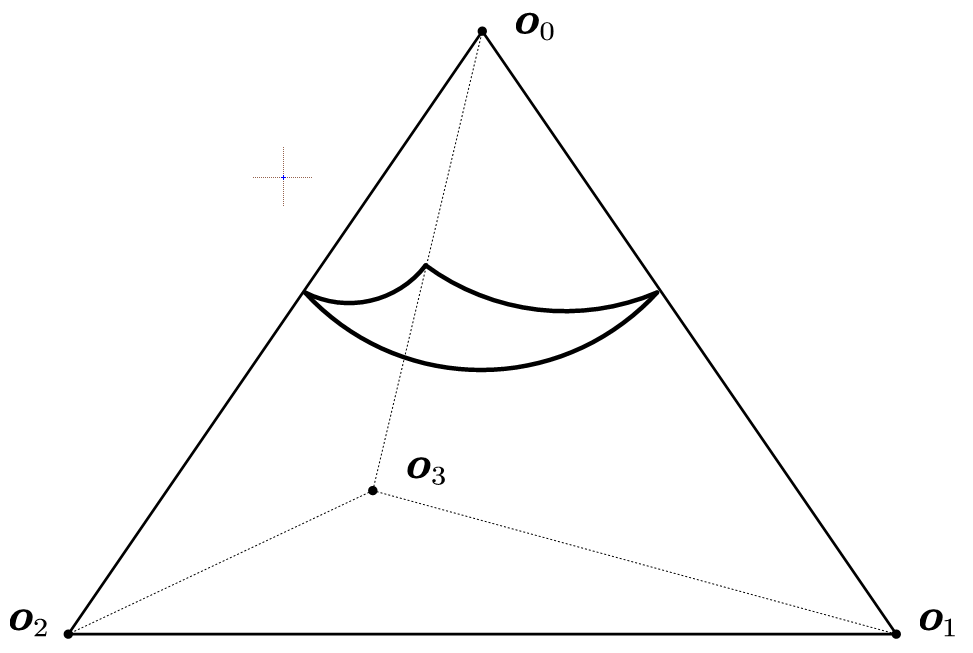}    &   \includegraphics[width=0.45\textwidth]{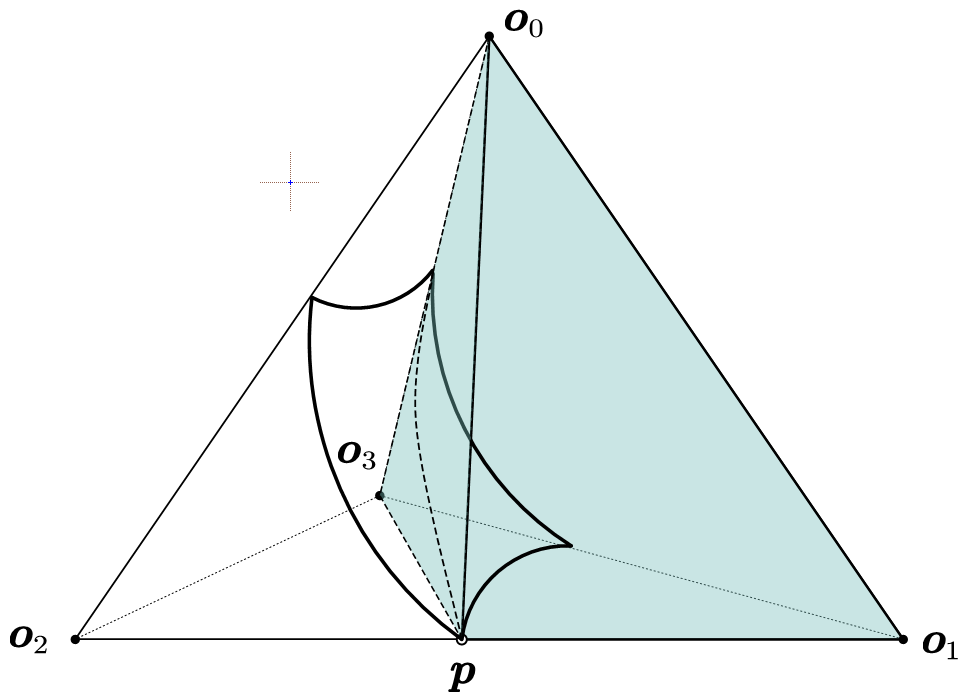}\\
           case 2  & case 3 
        \end{tabular}
     \caption{Illustration of \(\Gamma \cap T\) in Theorem \ref{thm:main}.}
    \label{fig:cut-3D}
\end{figure}

Similarly, for a smooth compact curve \(\gamma\subset\mathbb R^2\), let
\[
K_\gamma=\max_{x\in\gamma}|\kappa(x)|,
\]
where \(\kappa\) is the curvature of \(\gamma\), and let \(r_\gamma\) denote
its reach.

\begin{definition}\label{def:consistent2d}
A triangle \(T=\triangle\bm o_0\bm o_1\bm o_2\) is
\textbf{consistent} with \(\gamma\) if, for a constant \(c_0>0\) independent
of \(h_T=\operatorname{diam}(T)\),
\[
h_T< r_\gamma,\qquad
h_TK_\gamma\le\frac14,
\]
and
\[
\operatorname{dist}(\bm o_i,\gamma)
>
\max\{c_0h_T,K_\gamma h_T^2\},
\qquad i=0,1,2.
\]
\end{definition}

\begin{theorem}
\(\gamma = \{x \in \mathbb{R}^2: F(x) = 0\}\) is a smooth compact curve.
Let \(T=\triangle {\bm{o}_0\bm{o}_1\bm{o}_2}\) be a triangle consistent
with \(\gamma\) in the sense of Definition~\ref{def:consistent2d}.
\begin{itemize}
    \item[] Case 1: $$\text{\rm sgn}(F(\bm{o}_0))=\text{\rm sgn}(F(\bm{o}_1))=\text{\rm sgn}(F(\bm{o}_2)).$$ 
    In this case $$\gamma \cap T = \emptyset.$$
    
    \item[(2)] Case 2: As shown in Fig. \ref{fig:cut-2D} 
    $$-\text{\rm sgn}(F(\bm{o}_0))=\text{\rm sgn}(F(\bm{o}_1))=\text{\rm sgn}(F(\bm{o}_2))$$ 
    In this case, define the map 
    \begin{align*}
        \varphi:\quad &T\cap \gamma \longrightarrow\;  \overline{\bm{o}_1\bm{o}_2}\\
    & \bm{p}\;\longmapsto \;\overline{\bm{o}_0 \bm{p}}\cap \overline{ \bm{o}_1\bm{o}_2}. 
    \end{align*}
    $\varphi$ is bijective.
\end{itemize}
\label{thm:2D}
\end{theorem}

 \begin{figure}[htbp]
     \centering
    \includegraphics[width=0.5\textwidth]{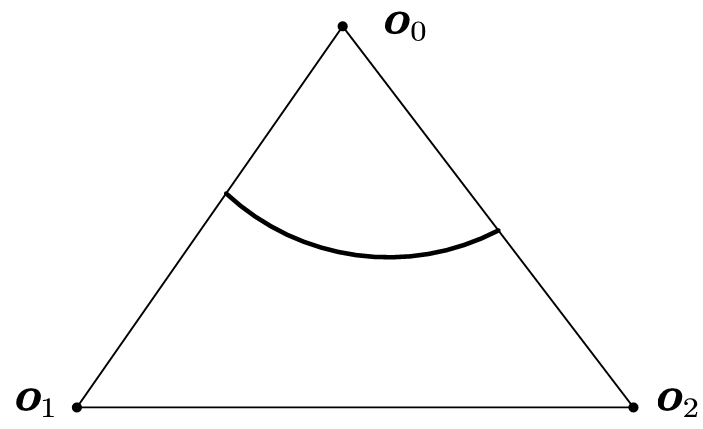}
     \caption{Illustration of \(\gamma \cap T\) in Theorem \ref{thm:2D}.}
     \label{fig:cut-2D}
 \end{figure}

\subsection{Vertex displacement for a consistent background mesh}
\label{subsec:vertex-displacement}

In this subsection we describe a simple mesh adjustment procedure which ensures that the consistent requirement in Definition \ref{def:consistent} is satisfied. In this step, vertices close to the zero level set are moved slightly away from the surface, so that no vertex of a cut tetrahedron lies too close to $\Gamma$. This avoids nearly degenerate intersection patterns.

Let
\[
\Gamma={x\in\mathbb R^3:F(x)=0}
\]

be a compact smooth regular level set. Let $d(x)$ denote the signed distance function to $\Gamma$ in a tubular neighborhood of $\Gamma$, and let
\[
n(x)=\nabla d(x)
\]
be the corresponding unit normal. Let $\mathcal T_h^0$ be an initial tetrahedral background mesh with mesh size $h$. We assume that $\mathcal T_h^0$ is shape-regular and quasi-uniform. More precisely, there exist constants $c_{\rm m},C_{\rm m}>0$, independent of $h$, such that for every tetrahedron $T\in\mathcal T_h^0$,
\[
c_{\rm m}h\le h_T\le C_{\rm m}h,
\]
where
$
h_T=\operatorname{diam}(T).
$
In practice, $\mathcal T_h^0$ is easy to obtain by subdividing a Cartesian background grid into tetrahedra.

Choose a constant $\rho>0$, independent of $h$, sufficiently small. For each vertex $a$ of $\mathcal T_h^0$, define the adjusted vertex $\widetilde a$ as follows.

If
$
|d(a)|\ge \rho h,
$
we leave the vertex unchanged, $\widetilde a=a.$

If
$
|d(a)|<\rho h,
$
we move the vertex along the normal direction to the offset level set $d=\pm \rho h$. Let $\pi_\Gamma(a)$ be the closest-point projection of $a$ onto $\Gamma$. Then we define
\[
\widetilde a
=
\pi_\Gamma(a)
+
\operatorname{sign}(d(a))\,\rho h\,n(\pi_\Gamma(a)).
\]
If $d(a)=0$, one may first apply an arbitrarily small generic translation of the background grid, so that no vertex lies exactly on $\Gamma$. Equivalently, one may assign a sign by a deterministic convention. Since vertex--surface coincidences are nongeneric, this case does not occur for a generic background grid.

The adjusted tetrahedral mesh $\mathcal T_h$ is obtained by replacing every vertex $a$ of $\mathcal T_h^0$ by $\widetilde a$, while keeping the same mesh connectivity.
The procedure can be summarized as Algorithm \ref{alg:vertex-displacement}.

\begin{algorithm}[t]
\caption{Vertex displacement away from the level-set surface}
\label{alg:vertex-displacement}
\begin{algorithmic}[1]
\Require Initial tetrahedral mesh $\mathcal T_h^0$ with mesh size $h$; level-set function $F$; parameter $\rho>0$.
\Ensure Adjusted tetrahedral mesh $\mathcal T_h$.

\For{each vertex $a$ of $\mathcal T_h^0$}
\State Compute the signed distance $d(a)$ to $\Gamma={x:F(x)=0}$.
\If{$|d(a)|\ge \rho h$}
\State Set $\widetilde a=a$.
\Else
\State Compute the closest-point projection $y=\pi_\Gamma(a)$.
\State Compute the unit normal $n(y)=\nabla d(y)$.
\State Set
\[
\widetilde a
=
y+\operatorname{sign}(d(a))\,\rho h\,n(y).
\]
\EndIf
\EndFor

\State Construct $\mathcal T_h$ by replacing every vertex $a$ of $\mathcal T_h^0$ by $\widetilde a$ and keeping the same tetrahedral connectivity.

\end{algorithmic}
\end{algorithm}


We next prove the geometric properties of the adjusted mesh.  The argument is
stated for simplices in dimension \(m=2\) or \(m=3\); in the present section
\(m=3\).

\begin{lemma}
\label{lem:shape-regularity-after-displacement}
Suppose that the initial mesh \(\mathcal T_h^0\) is shape-regular and
quasi-uniform.  There exists \(\rho_*>0\), depending only on the
shape-regularity constants of \(\mathcal T_h^0\), such that, if
\(0<\rho\le\rho_*\), the adjusted mesh \(\mathcal T_h\) remains conforming,
shape-regular, and quasi-uniform.  No simplex is inverted, and there are
constants \(c_1,C_1>0\), independent of \(h\), such that
\[
 c_1h\le h_T\le C_1h
 \qquad\text{for every }T\in\mathcal T_h .
\]
\end{lemma}

\begin{proof}
Let \(T^0=\operatorname{conv}\{a_0,\ldots,a_m\}\) be an initial simplex and
write its edge matrix as
\[
B_{T^0}
=
\begin{bmatrix}
a_1-a_0&\cdots&a_m-a_0
\end{bmatrix}.
\]
Shape-regularity and quasi-uniformity are equivalent to the existence of
constants \(c_s,C_s>0\), independent of \(T^0\) and \(h\), such that
\[
c_sh\le \sigma_{\min}(B_{T^0})
\le \sigma_{\max}(B_{T^0})\le C_sh .
\]
Set \(\delta_i=\widetilde a_i-a_i\).  The signed-distance construction gives
\(|\delta_i|\le\rho h\).  The edge matrix of the adjusted simplex is
\[
B_T=B_{T^0}+E_T,\qquad
E_T=
\begin{bmatrix}
\delta_1-\delta_0&\cdots&\delta_m-\delta_0
\end{bmatrix},
\]
and hence
\[
\|E_T\|_2\le 2\sqrt m\,\rho h .
\]
Weyl's singular-value perturbation inequality therefore gives
\[
\sigma_{\min}(B_T)
\ge(c_s-2\sqrt m\,\rho)h,
\qquad
\sigma_{\max}(B_T)
\le(C_s+2\sqrt m\,\rho)h .
\]
Choose, for example,
\(\rho_*=c_s/(4\sqrt m)\).  Then
\[
\frac{c_s}{2}h
\le\sigma_{\min}(B_T)
\le\sigma_{\max}(B_T)
\le\left(C_s+\frac{c_s}{2}\right)h .
\]
    The same estimate holds for the homotopy
    \(B_{T^0}+tE_T\), \(0\le t\le1\).  Its determinant never vanishes, so its
    orientation cannot change; consequently no simplex is inverted.

    For completeness, the local determinant argument also gives a global mesh
    embedding.  Let \(u_h\) be the continuous piecewise-affine function whose
    nodal values are \(\delta_i\), and set \(\Phi_h(x)=x+u_h(x)\).  On
    \(T^0\),
    \[
    D u_h=E_TB_{T^0}^{-1},
    \qquad
    \|D u_h\|_2
    \le \frac{2\sqrt m\,\rho}{c_s}
    \le\frac12 .
    \]
    The computational domain is a convex polygonal or polyhedral box, so the
    piecewise-affine function \(u_h\) is globally Lipschitz with the same
    bound.  Hence, for all \(x,y\) in the box,
    \[
    |\Phi_h(x)-\Phi_h(y)|
    \ge
    \left(1-\frac{2\sqrt m\,\rho}{c_s}\right)|x-y|
    \ge\frac12|x-y|.
    \]
    Thus \(\Phi_h\) is injective.  Since a shared vertex is moved only once
    and the connectivity is unchanged, adjacent simplices retain their common
    faces, while nonadjacent simplices cannot overlap.  The adjusted
    simplicial complex is therefore conforming.  The singular-value bounds
    prove uniform shape-regularity and quasi-uniformity, as well as the
    asserted diameter estimates.
\end{proof}

\begin{lemma}
\label{lem:vertex-distance-after-displacement}
Let \(T\in\mathcal T_h\) be a cut tetrahedron.  There are constants
\(c_0,C_0>0\), independent of \(h\), such that every vertex \(a\) of \(T\)
satisfies
\begin{equation}
2c_0h_T\le |d(a)|\le C_0h_T .
\label{eq:vertex-distance-condition}
\end{equation}
\end{lemma}

\begin{proof}
Every adjusted mesh vertex satisfies \(|d(a)|\ge\rho h\): an unchanged vertex
has this property by the displacement criterion, while a moved vertex lies on
the offset surface \(|d|=\rho h\).  By
Lemma~\ref{lem:shape-regularity-after-displacement},
\(h_T\le C_1h\).  Therefore
\[
|d(a)|\ge\rho h\ge\frac{\rho}{C_1}h_T.
\]
Taking \(c_0=\rho/(2C_1)\) gives the lower bound in
\eqref{eq:vertex-distance-condition}.  Since \(T\) is cut, choose
\(z\in T\cap\Gamma\).  The signed distance is \(1\)-Lipschitz, and hence
\[
|d(a)|\le |a-z|\le h_T.
\]
Thus one may take \(C_0=1\).
\end{proof}

\begin{proposition}
\label{prop:displacement-implies-consistency}
Let \(\Gamma\) be a compact smooth surface with reach \(r_\Gamma>0\), and let
\(K_\Gamma\) be the curvature constant defined above.  If \(h\) is
sufficiently small, then every cut tetrahedron of the adjusted mesh is
consistent with \(\Gamma\) in the sense of Definition~\ref{def:consistent}.
The same assertion holds for triangles and
Definition~\ref{def:consistent2d}.
\end{proposition}

\begin{proof}
By Lemma~\ref{lem:shape-regularity-after-displacement},
\(\max_T h_T\le C_1h\).  Thus, for sufficiently small \(h\),
\[
h_T<r_\Gamma,\qquad
h_TK_\Gamma\le\frac14,\qquad
K_\Gamma h_T<2c_0
\]
for every adjusted simplex; the last condition is void if \(K_\Gamma=0\).
For each vertex \(a\) of a cut simplex,
Lemma~\ref{lem:vertex-distance-after-displacement} gives
\[
\operatorname{dist}(a,\Gamma)
=|d(a)|
\ge2c_0h_T
>
\max\{c_0h_T,K_\Gamma h_T^2\}.
\]
All requirements in the corresponding consistency definition are therefore
satisfied.
\end{proof}

\paragraph{Direct displacement using a general level-set function.}
The signed-distance construction is convenient for the proof and gives an
exact prescribed offset.  It is not necessary, however, to compute the closest
point when the level-set function \(F\) itself is available.  We now give a
direct alternative and show that it has the same mesh-quality consequence.
Let
\[
\mathcal N_\delta=\{x:\operatorname{dist}(x,\Gamma)<\delta\}
\]
be a fixed tubular neighborhood.  Since \(F\in C^2\) and
\(\nabla F\ne0\) on the compact surface \(\Gamma\), after decreasing
\(\delta\) if needed there exist positive constants
\(m_F,M_F,L_F,c_F,C_F\) such that, for \(x\in\mathcal N_\delta\),
\begin{equation}
m_F\le\|\nabla F(x)\|\le M_F,\qquad
\|D^2F(x)\|\le L_F,\qquad
c_F|d(x)|\le |F(x)|\le C_F|d(x)|.
\label{eq:F-distance-equivalence}
\end{equation}

Choose a sufficiently small constant \(\rho_g>0\).  For a mesh vertex \(a\)
in \(\mathcal N_\delta\), define
\[
\delta_F(a)=\frac{|F(a)|}{\|\nabla F(a)\|}.
\]
If \(\delta_F(a)\ge\rho_gH\), leave the vertex unchanged.  Otherwise set
\begin{equation}
\widetilde a
=
a+2\rho_gh\,
\frac{\operatorname{sign}(F(a))\nabla F(a)}
     {\|\nabla F(a)\|}.
\label{eq:direct-gradient-displacement}
\end{equation}
When \(F(a)=0\), either sign may be assigned by a fixed deterministic
convention.  Vertices outside \(\mathcal N_\delta\) are left unchanged.  If
\(\Gamma\) is compactly contained in the computational box, then for
sufficiently small \(h\) this operation does not move boundary vertices.

\begin{proposition}
\label{prop:gradient-displacement}
For sufficiently small \(\rho_g\) and \(h\), the direct gradient displacement
\eqref{eq:direct-gradient-displacement} produces a conforming,
shape-regular, and quasi-uniform mesh.  Moreover, every adjusted vertex in the
active band satisfies
\[
\operatorname{dist}(\widetilde a,\Gamma)\ge c_gh
\]
with \(c_g>0\) independent of \(h\).  Consequently, for sufficiently small
\(h\), every cut simplex is consistent with \(\Gamma\).
\end{proposition}

\begin{proof}
The displacement of each vertex has length at most \(2\rho_gh\).  Repeating
the edge-matrix argument in
Lemma~\ref{lem:shape-regularity-after-displacement}, now with
\(\|E_T\|_2\le4\sqrt m\,\rho_gH\), proves shape-regularity,
quasi-uniformity, conformity, and preservation of orientation whenever
\(\rho_g<c_s/(8\sqrt m)\).

It remains to prove separation from \(\Gamma\).  An unchanged vertex in the
active band satisfies
\[
|F(a)|\ge\rho_gh\|\nabla F(a)\|
\ge\rho_gm_Fh.
\]
For a moved vertex set
\[
v_a=\frac{\operatorname{sign}(F(a))\nabla F(a)}
          {\|\nabla F(a)\|}.
\]
Taylor's theorem along the segment from \(a\) to
\(\widetilde a=a+2\rho_gHv_a\) gives
\[
\operatorname{sign}(F(a))F(\widetilde a)
=
|F(a)|+2\rho_gh\|\nabla F(a)\|+R_a,
\qquad
|R_a|\le 2L_F\rho_g^2h^2 .
\]
The same formula holds with the assigned sign when \(F(a)=0\).  For
sufficiently small \(h\),
\[
|F(\widetilde a)|\ge \rho_gm_Fh.
\]
Equation~\eqref{eq:F-distance-equivalence} then yields
\[
\operatorname{dist}(\widetilde a,\Gamma)
\ge \frac{\rho_gm_F}{C_F}h.
\]
    The same use of \eqref{eq:F-distance-equivalence} applies to the unchanged
    vertices considered above.  Thus all vertices in the tubular band have
    distance at least \(c_gh\), with \(c_g=\rho_gm_F/C_F\).
    If a simplex is cut by \(\Gamma\), every one of its vertices is within
    \(h_T\) of \(\Gamma\); hence, for sufficiently small \(h\), all of its
    vertices belong to \(\mathcal N_\delta\).  The proof of
    Lemma~\ref{lem:vertex-distance-after-displacement} and
    Proposition~\ref{prop:displacement-implies-consistency} now applies verbatim,
    with \(c_g\) in place of \(\rho\).
\end{proof}

This displacement procedure actually implies a uniform transversality property for the regularity analysis in Section \ref{sec:order}. Let $a$ be a vertex of a cut tetrahedron used as the apex in the local parametrization, and let $z\in T\cap\Gamma$. Since $d(z)=0$, Taylor expansion of the signed distance function gives
\[
d(a)
=
\nabla d(z)\cdot(a-z)+O(h_T^2)
=
n(z)\cdot(a-z)+O(h_T^2).
\]
The displacement procedure gives,
\[
|d(a)|\ge c_0h_T.
\]
For sufficiently small $h_T$, the quadratic remainder is bounded by $(c_0/2)h_T$, and hence
\[
|n(z)\cdot(a-z)|
\ge
\frac{c_0}{2}h_T.
\]
Since $|a-z|\le h_T$, it follows that
\begin{equation}
\left|
n(z)\cdot\frac{a-z}{|a-z|}
\right|
\ge c_2>0,
\label{eq:uniform-transversality}    
\end{equation}
where $c_2$ is independent of $h$. Thus the direction from the apex vertex to the surface is uniformly transverse to the tangent plane of $\Gamma$. This prevents the denominators appearing in the local change-of-variables formula from becoming singular and is important in the regularity analysis in Section \ref{sec:order}.


\section{2D Parametrization}

\subsection{Parametrization of Curves}
\label{sec:2dcurve}

Let \(T=\triangle \bm{o}_0\bm{o}_1\bm{o}_2\) be consistent with \(\Gamma\).
By Theorem \ref{thm:2D}, \(\Gamma \cap T\) has a direct parametrization; see
Fig. \ref{fig:para-2D}.

\begin{figure}[htbp]
    \centering
    \includegraphics[width=0.5\textwidth]{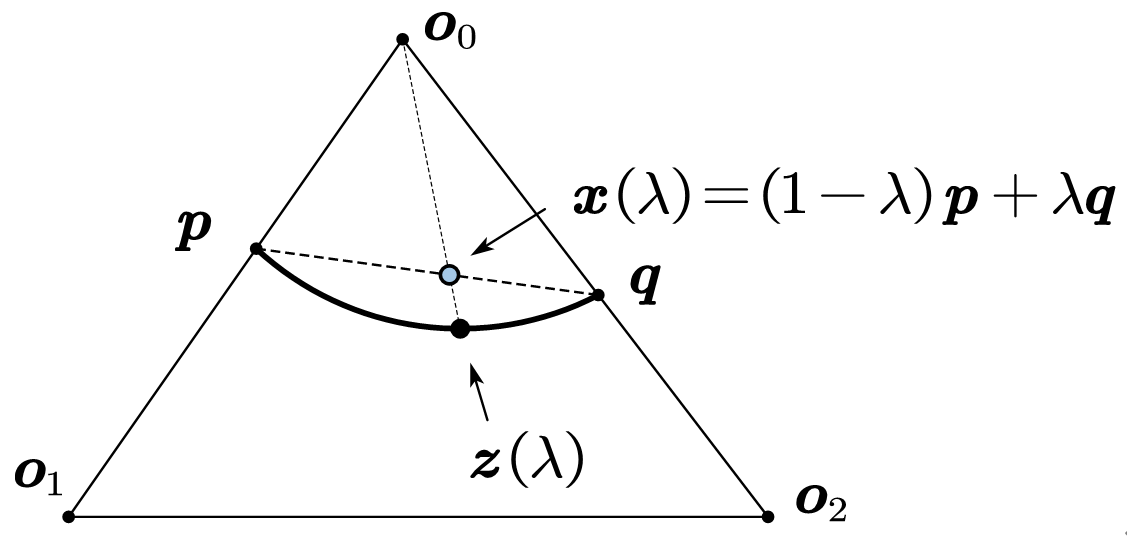}
    \caption{Parametrization of \(\Gamma \cap T\).}
    \label{fig:para-2D}
\end{figure}

Without loss of generality, we assume $\bm{o}_0=(0,0)$. As shown in Fig. \ref{fig:para-2D}, 
$\bm{z}(\lambda),\;\lambda\in [0,1]$  gives a natural parametrization of \(T\cap \Gamma\).
With this parametrization, the integral over $T\cap \Gamma$ can be transformed to an integral over $[0,1]$.
\[
\int_{T\cap \Gamma} f(x) \, ds = \int_0^1 f(\bm{z}(\lambda)) |\bm{z}_\lambda(\lambda)| d\lambda,
\]
The standard Gauss-Legendre quadrature applies to compute the integral with high order accuracy.

In the implementation of Gauss-Legendre quadrature, we need to compute $\bm{z}(\lambda)$ and $\bm{z}_\lambda(\lambda)$ at Gauss-Legendre points. With any given $\lambda\in [0,1]$, $\bm{z}(\lambda)$ can be obtained by solving a nonlinear equation.
First, we solve $\alpha(\lambda)$,
\[
 F(\alpha(\lambda)\bm{x}(\lambda))=0,\quad \bm{x}(\lambda) = (1-\lambda)\bm{p} + \lambda\bm{q}
\]
Then $\bm{z}(\lambda)$ is given by 
\begin{align}
    \bm{z}(\lambda)=&\alpha(\lambda)\bm{x}(\lambda) 
    \label{eq:z}
\end{align}
With $\alpha(\lambda)$, the computation of $\bm{z}_\lambda(\lambda)$ is straightforward. 
Using $F(\bm{z}(\lambda))=0$ and \eqref{eq:z}, we have
\begin{align*}   \bm{z}_\lambda(\lambda)=&\alpha_\lambda(\lambda)\bm{x}(\lambda)+\alpha(\lambda)\bm{x}_\lambda(\lambda),\quad   \nabla F(\bm{z}(\lambda))\cdot \bm{z}_\lambda(\lambda)=0, 
\end{align*}
and 
\begin{align*}
\alpha_\lambda(\lambda)=-\frac{\alpha(\lambda)\bm{x}_\lambda(\lambda)\cdot \nabla F(\bm{z}(\lambda))}{\bm{x}(\lambda)\cdot \nabla F(\bm{z}(\lambda))}.    
\end{align*}

\subsection{Parametrization of 2D Areas}
\label{sec:2dregion}
Now we turn to consider the integration over the bounded region surrounded by $\Gamma$,  
$$\int_\Omega f(x){\rm d} x,$$
with $\Omega=\{\bm{x}\in \mathbb{R}^2: F(\bm{x})\le 0\}$.
 
Let \(U\) be a polygonal computational box with
\(\overline\Omega\subset U\), and let
\(\mathcal T_h=\{T_i\}\) be a conforming triangulation of \(U\).  We require
consistency only for triangles cut by \(\Gamma\).  Then
\[
\int_\Omega f(\bm{x}) \, d\bm{x} = \sum_{i} \int_{\Omega \cap T_i} f(\bm{x}) \, {\rm d}\bm{x}.
\]
 
The three types of elements are treated as follows.  If \(T_i\subset\Omega\),
we apply the affine pullback of a fixed reference-triangle rule that is exact
for polynomials of total degree at most \(2N_q-1\).  One such rule is obtained
from the two-dimensional Duffy map by using \(N_q+1\) Gauss--Legendre nodes
in the collapsed coordinate and \(N_q\) nodes in the other coordinate.  If
\(T_i\cap\overline\Omega=\emptyset\), its contribution is zero.  On a cut
triangle, Theorem~\ref{thm:2D} gives the local parametrization described
below.  If the exceptional vertex lies outside \(\Omega\), the same
parametrization is applied to \(\Omega^c\cap T_i\), and we use
\[
\int_{\Omega \cap T_i} f(\bm{x}) \, {\rm d}\bm{x}
=
\int_{T_i} f(\bm{x}) \, {\rm d}\bm{x}
-
\int_{\Omega^c \cap T_i} f(\bm{x}) \, {\rm d}\bm{x}.
\]
The whole-element integral in this subtraction is evaluated with the same
reference-triangle rule as an interior element.
\begin{figure}[htbp]
    \centering
    \includegraphics[width=0.5\textwidth]{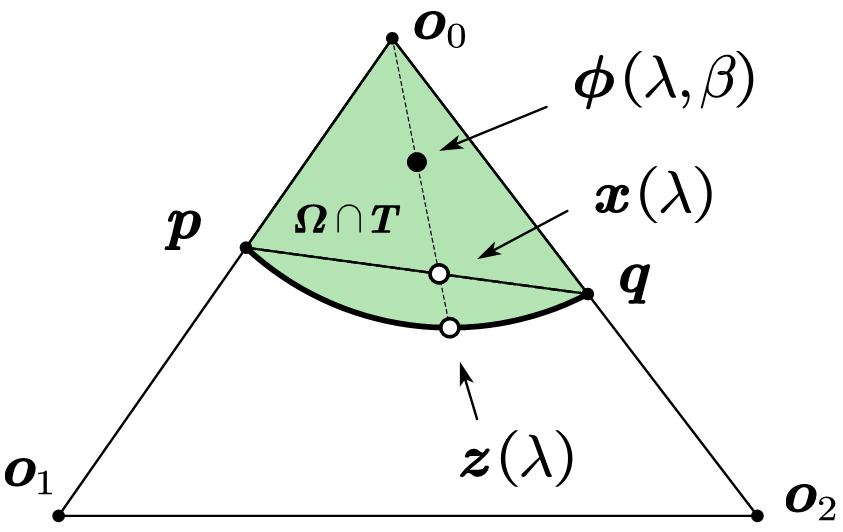}
    \caption{Parametrization of \(\Omega \cap T\) in $\mathbb{R}^2$.}
    \label{fig:para-region-2D}
\end{figure}
As shown in Fig. \ref{fig:para-region-2D}, the parametrization of $\Omega\cap T$ is straightforward.
\begin{align}
    \bm{\phi}(\lambda,\beta)=\beta \bm{z}(\lambda),\quad \beta\in [0,1].
\end{align}

By change of variable formula, we have
$$
\int_{\Omega\cap T} f
=
\int_{0}^1\int_0^1
\left|\operatorname{det} J(\bm{\phi})\right|
f(\bm{\phi}(\lambda,\beta))\,d\beta\,d\lambda
$$
where $J(\phi)$ denotes the Jacobi matrix of $\bm{\phi}$,
$$J(\bm{\phi})=\left[\beta\bm{z}_\lambda(\lambda),\bm{z}(\lambda)\right]\in \mathbb{R}^{2\times 2}$$

Also, standard Gauss-Legendre quadrature rule is applied to compute the integral. The computation of $\bm{\phi}(\lambda,\beta)$ and $J(\bm{\phi})$ follows the same procedure as in the previous subsection.

\section{3D Parametrization}
\subsection{Parametrization of Surfaces}
\label{sec:3dplane}

We now introduce the numerical method to compute the surface integral
over
\[
    \Gamma=\{\bm{x}\in \mathbb{R}^3:F(\bm{x})=0\}.
\]
Assume that the surface is covered by a tetrahedral mesh,
\[
    \Gamma\subset \mathcal{T}=\bigcup_i T_i,
\]
and that each cut tetrahedron is consistent with $\Gamma$ in the sense of
Definition~\ref{def:consistent}. Then the surface integral is decomposed into
local contributions
\[
    \int_{\Gamma} f\,dS
    =
    \sum_i \int_{T_i\cap\Gamma} f\,dS .
\]

By Theorem~\ref{thm:main}, each cut tetrahedron has a simple intersection
pattern with $\Gamma$. Case 1 is trivial. Case 3 can be reduced to Case 2 by
splitting the tetrahedron into two sub-tetrahedra, as shown in
Fig.~\ref{fig:cut-3D}. Hence it suffices to describe the parametrization in
Case 2.

Assume, after a translation of coordinates, that the exceptional vertex is
\[
    \bm{o}_0=(0,0,0).
\]
Let $\bm p,\bm q,\bm r$ be the three intersection points between $\Gamma$ and
the three edges connecting $\bm{o}_0$ to the opposite face. More precisely,
after a suitable relabeling,
\[
    \bm p\in \overline{\bm{o}_0\bm{o}_1}\cap\Gamma,\qquad
    \bm q\in \overline{\bm{o}_0\bm{o}_2}\cap\Gamma,\qquad
    \bm r\in \overline{\bm{o}_0\bm{o}_3}\cap\Gamma .
\]
Theorem~\ref{thm:main} guarantees that these intersections are unique.
 \begin{figure}[htbp]
    \centering
\includegraphics[width=0.5\textwidth]{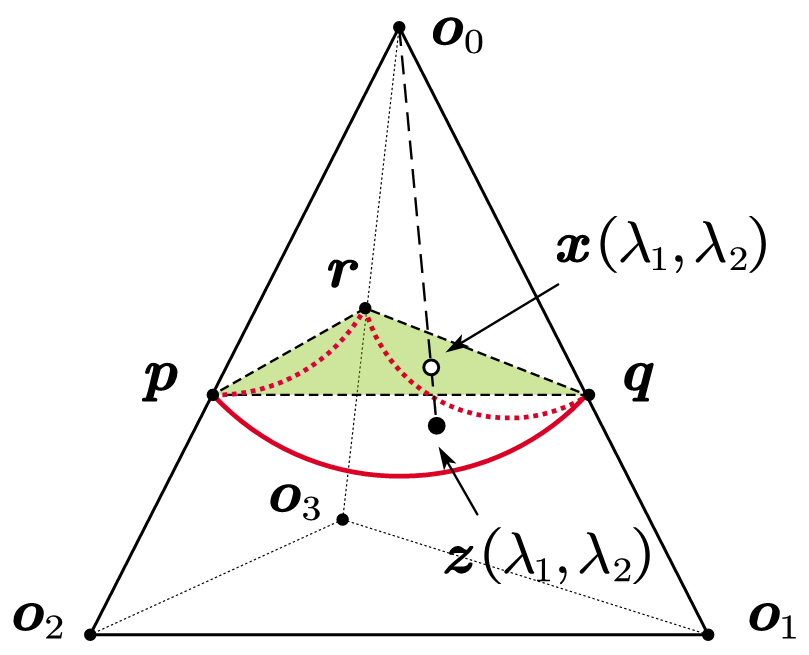}
     \caption{Parametrization of \(\Gamma \cap T\) in $\mathbb{R}^3$.}
     \label{fig:para-3D}
     \end{figure}
We use the reference triangle
\[
    \widehat K
    =
    \left\{
    \bm\lambda=(\lambda_1,\lambda_2)\in\mathbb{R}^2:
    \lambda_1\ge0,\ \lambda_2\ge0,\ \lambda_1+\lambda_2\le1
    \right\}.
\]
On $\widehat K$, define the affine chord-triangle parametrization
\begin{equation}
\label{eq:x-triangular-3d}
    \bm{x}(\lambda_1,\lambda_2)
    =
    (1-\lambda_1-\lambda_2)\bm p
    +\lambda_1\bm q
    +\lambda_2\bm r .
\end{equation}
 In particular,
\[
    \partial_{\lambda_1}\bm{x}=\bm q-\bm p,\qquad
    \partial_{\lambda_2}\bm{x}=\bm r-\bm p,\qquad
    \partial_{\bm\lambda}^{\nu}\bm{x}=0
    \quad\text{for all }|\nu|\ge2 .
\]

For each $\bm\lambda\in\widehat K$, the ray
\[
    \{\alpha\bm{x}(\bm\lambda):\alpha>0\}
\]
intersects $\Gamma$ exactly once. Therefore there exists a unique positive
function $\alpha(\bm\lambda)$ such that
\begin{equation}
\label{eq:alpha-equation-triangle}
    F\bigl(\alpha(\bm\lambda)\bm{x}(\bm\lambda)\bigr)=0.
\end{equation}
We define
\begin{equation}
\label{eq:z-triangular-3d}
    \bm{z}(\bm\lambda)
    =
    \alpha(\bm\lambda)\bm{x}(\bm\lambda),
    \qquad
    \bm\lambda\in\widehat K .
\end{equation}
Then $\bm z:\widehat K\to T\cap\Gamma$ gives a parametrization of the surface
patch.

The surface integral over one cut tetrahedron is therefore
\begin{equation}
\label{eq:surface-integral-triangle}
    \int_{T\cap\Gamma} f(\bm{x})\,dS
    =
    \int_{\widehat K}
    f(\bm z(\bm\lambda))
    \left|
        \partial_{\lambda_1}\bm z(\bm\lambda)
        \times
        \partial_{\lambda_2}\bm z(\bm\lambda)
    \right|
    \,d\bm\lambda .
\end{equation}

At quadrature nodes, $\bm z(\bm\lambda)$ is obtained by solving the scalar
nonlinear equation~\eqref{eq:alpha-equation-triangle}. Its first derivatives are
computed by differentiating
\[
    F(\bm z(\bm\lambda))=0.
\]
For $i=1,2$, we have
\begin{equation}
\label{eq:z-derivative-triangle}
    \partial_{\lambda_i}\bm z
    =
    \partial_{\lambda_i}\alpha\,\bm x
    +
    \alpha\,\partial_{\lambda_i}\bm x ,
\end{equation}
and
\begin{equation}
\label{eq:alpha-derivative-triangle}
    \partial_{\lambda_i}\alpha
    =
    -
    \frac{
        \alpha\,
        \partial_{\lambda_i}\bm x\cdot \nabla F(\bm z)
    }{
        \bm x\cdot\nabla F(\bm z)
    } .
\end{equation}
The denominator is uniformly bounded away from zero in the following scaled
sense:
\[
    |\bm x(\bm\lambda)\cdot\nabla F(\bm z(\bm\lambda))|
    \ge c h_T ,
\]
where $c>0$ is independent of $h_T$. This follows from the uniform
transversality estimate~\eqref{eq:uniform-transversality} and from
$\nabla F=|\nabla F|\bm n$ on $\Gamma$.

Finally, we evaluate~\eqref{eq:surface-integral-triangle} by first applying the
Duffy transformation
\[
\mathcal D(\xi,\eta)
=
\bigl(\xi,(1-\xi)\eta\bigr),
\qquad
(\xi,\eta)\in[0,1]^2,
\]
whose Jacobian determinant is \(1-\xi\). Thus
\[
\int_{\widehat K}g(\bm\lambda)\,d\bm\lambda
=
\int_0^1\int_0^1
g\bigl(\xi,(1-\xi)\eta\bigr)(1-\xi)\,d\eta\,d\xi.
\]
Let \(\{\xi_i,\omega_i\}_{i=1}^{N_q}\) be the Gauss--Legendre nodes and
weights on \([0,1]\). The rule used in the computations is the tensor-product
rule
\[
\int_{\widehat K}g(\bm\lambda)\,d\bm\lambda
\approx
\sum_{i=1}^{N_q}\sum_{j=1}^{N_q}
\omega_i\omega_j(1-\xi_i)
g\bigl(\xi_i,(1-\xi_i)\xi_j\bigr).
\]
Because the Gauss--Legendre nodes lie in \((0,1)\), all weights in this
triangle rule are strictly positive. We denote the induced triangle rule by
\(Q_{\widehat K}^{N_q}\).

\subsection{Parametrization of 3D Region}
\label{sec:3dregion}

Regarding the integral over the bounded region,
$$\Omega=\{\bm{x}\in \mathbb{R}^3: F(\bm{x})\le 0\},$$
we use a tetrahedral mesh of a box \(U\) satisfying
\(\overline\Omega\subset U\).  An uncut tetrahedron contained in \(\Omega\)
is integrated by an affine pullback of a fixed reference-tetrahedron rule
that is exact for polynomials of total degree at most \(2N_q-1\),
and an exterior tetrahedron contributes zero.  The reference rule is obtained
from the Duffy map
\[
(\xi,\eta,\zeta)
\longmapsto
\bigl(\xi,(1-\xi)\eta,(1-\xi)(1-\eta)\zeta\bigr),
\qquad(\xi,\eta,\zeta)\in[0,1]^3,
\]
whose Jacobian is \((1-\xi)^2(1-\eta)\), followed by a tensor-product
Gauss--Legendre rule with \(N_q+1,N_q+1,N_q\) nodes in
\(\xi,\eta,\zeta\), respectively.  On a cut tetrahedron, the cone parametrization below
is used.  If it parametrizes the exterior portion, the desired contribution
is the whole-tetrahedron rule minus the exterior cone rule.  Thus, as in the
two-dimensional case, it suffices to describe the configuration shown in
Fig.~\ref{fig:para-region-3D}.
\begin{figure}[htbp]
    \centering
\includegraphics[width=0.5\textwidth]{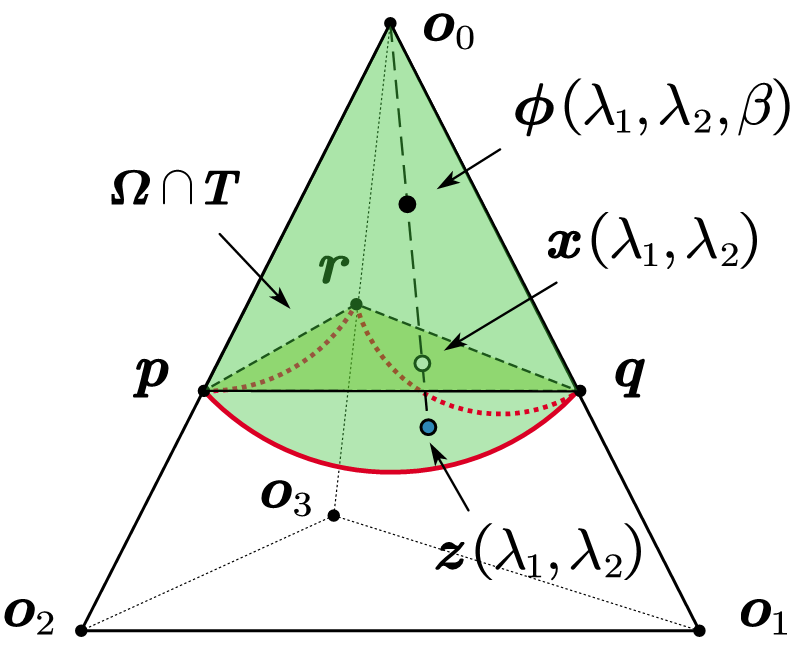}
    \caption{Parametrization of \(\Omega \cap T\) in $\mathbb{R}^3$.}
    \label{fig:para-region-3D}
\end{figure}

The parametrization is also very easy to construct. 
\begin{align}
    \bm{\phi}(\lambda_1,\lambda_2,\beta)
    =
    \beta \bm{z}(\lambda_1,\lambda_2),
    \qquad
    (\lambda_1,\lambda_2)\in\widehat K,\quad \beta\in[0,1].
\end{align}

By change of variable formula, we have
$$
\int_{\Omega\cap T} f
=
\int_{\widehat K}\int_0^1
\left|\operatorname{det}J(\bm{\phi})\right|
f(\bm{\phi}(\lambda_1,\lambda_2,\beta))
\,d\beta\,d\bm\lambda
$$
with $J(\bm{\phi})$ is also the Jacobi matrix of $\bm{\phi}$,
$$J(\bm{\phi})=\left[\beta\partial_{\lambda_1}\bm{z}(\lambda_1,\lambda_2),\beta\partial_{\lambda_2}\bm{z}(\lambda_1,\lambda_2),\bm{z}(\lambda_1,\lambda_2)\right]\in \mathbb{R}^{3\times 3}.$$
The \(\widehat K\)-integration is evaluated with the Duffy-transformed
tensor-product Gauss--Legendre rule described in
Section~\ref{sec:3dplane}, and a one-dimensional Gauss--Legendre rule is used
for \(\beta\).

\section{Regularity of the Parametrization}
\label{sec:order}
\subsection{Error over Curve}

In this section we establish uniform regularity estimates for the local
parametrizations introduced in Sections~3 and~4. Throughout the analysis, we
assume that the background meshes are shape-regular and quasi-uniform, with
the consistency and transversality constants uniform in \(h\). We also assume
that every one-dimensional nonlinear equation defining an intersection point
or a radial factor \(\alpha\) is solved by a safeguarded scalar solver to
machine precision.  The resulting roundoff-level root error is treated as
negligible in the \(h\)-asymptotic estimates, so the computed roots are
identified with the exact roots below.  Thus the estimates measure the
quadrature error above the machine-precision floor.  The key point is that the derivatives of the
parametrizations, measured on fixed reference domains, scale with the diameter
of the cut element in the expected way.

To streamline our induction proofs, we first introduce a basic algebraic lemma regarding integer arithmetic.
\begin{lemma}[Algebraic Inequality]
\label{lem:algebraic}
    For any integer $m \geq 2$ and positive integers $x_1, x_2, \ldots, x_m$, we have
    \[
    \sum_{i=1}^m \left\lfloor\frac{x_i-1}{2}\right\rfloor
    \le
    \left\lfloor\frac{\sum_{i=1}^m x_i}{2}\right\rfloor - 1 .
    \]
\end{lemma}

We now introduce the notation for the 2D curve case in the same form as in
Section \ref{sec:2dcurve}. Let
\(T=\triangle\bm o_0\bm o_1\bm o_2\) be a triangle consistent with
\(\Gamma\), assume that \(\bm o_0\) is the exceptional vertex, and let
\[
\bm p\in\overline{\bm o_0\bm o_1}\cap\Gamma,
\qquad
\bm q\in\overline{\bm o_0\bm o_2}\cap\Gamma
\]
be the unique edge-intersection points. Assume, without loss of generality,
that \(\bm{o}_0=(0,0)\). We use
\[
    \bm{x}(\lambda)=(1-\lambda)\bm{p}+\lambda\bm{q},\qquad
    \bm{z}(\lambda)=\alpha(\lambda)\bm{x}(\lambda),\qquad
    \lambda\in[0,1],
\]
where $\bm{z}(\lambda)$ is determined by
\[
    F(\bm{z}(\lambda))=F(\alpha(\lambda)\bm{x}(\lambda))=0.
\]
Here $h$ denotes the typical mesh size, and hence
\[
    |\bm{x}(\lambda)|+|\partial_\lambda\bm{x}(\lambda)|\le Ch,
    \qquad
    \partial_\lambda^j\bm{x}(\lambda)=0\quad (j\ge 2).
\]
With this parametrization, the integral over $T\cap\Gamma$ is written as
\[
\int_{T\cap \Gamma} f(\bm{x})\,ds
=
\int_0^1 f(\bm{z}(\lambda))|\partial_\lambda\bm{z}(\lambda)|\,d\lambda .
\]
Using the error estimate of the Gauss-Legendre quadrature rule with $N_q$ nodes, the local quadrature error is controlled by
\[
e_{curve,T}
=
O\left(
\sup_{\lambda\in[0,1]}
\left|
\frac{d^{2N_q}}{d\lambda^{2N_q}}
\bigl(
f(\bm{z}(\lambda))|\partial_\lambda\bm{z}(\lambda)|
\bigr)
\right|
\right).
\]

Function $f$ is assumed to be smooth enough, so we only need to bound the derivatives of the parameterization.

    Let $\bm{\gamma}(\sigma)$ be the arc-length parametrization of $\Gamma$ on $T\cap\Gamma$ such that
    \[
        \bm{z}(\lambda)=\bm{\gamma}(\sigma(\lambda)).
    \]
    Since $\bm{z}(\lambda)=\alpha(\lambda)\bm{x}(\lambda)$, we write
    \[
        \rho(\lambda)=\alpha(\lambda)-1,
    \]
    and the relation between the approximating segment $\bm{x}(\lambda)$ and the exact curve can be written as
    \begin{equation}
    \label{eq:restriction2d}
        \bm{\gamma}(\sigma(\lambda))-(1+\rho(\lambda))\bm{x}(\lambda)=0 .
    \end{equation}
    In this scale, $\bm{x}(\lambda)/h=O(1)$ and $\rho(\lambda)=O(h)$.
    We first record the nondegeneracy needed when differentiating the
    arc-length Jacobian.

    \begin{lemma}
    \label{lem:curve-speed-lower-bound}
    There are constants \(c,C>0\), independent of \(h\) and of the cut
    triangle, such that
    \[
    ch\le |\bm x(\lambda)|\le Ch,\qquad
    c\le\alpha(\lambda)\le C,\qquad
    ch\le|\partial_\lambda\bm z(\lambda)|\le Ch
    \]
    for every \(\lambda\in[0,1]\).
    \end{lemma}

    \begin{proof}
    The signed distance is \(1\)-Lipschitz.  Since
    \(\bm p\in\Gamma\cap\overline{\bm o_0\bm o_1}\), consistency gives
    \[
    |\bm p-\bm o_0|
    \ge |d(\bm o_0)|\ge c_0h,
    \qquad
    |\bm o_1-\bm p|
    \ge |d(\bm o_1)|\ge c_0h,
    \]
    and the analogous estimates hold for \(\bm q\) on
    \(\overline{\bm o_0\bm o_2}\).  Shape-regularity of \(T\) implies that
    its edge lengths are comparable to \(h\) and that the angle between the
    two edges issuing from \(\bm o_0\) is bounded away from \(0\) and
    \(\pi\).  Consequently,
    \[
    |\bm p|+|\bm q|+|\bm q-\bm p|\le Ch,\qquad
    |\det[\bm p,\bm q]|\ge ch^2,
    \]
    after translating \(\bm o_0\) to the origin.  The smallest singular
    value of the matrix \([\bm p,\bm q]\) is therefore comparable to \(h\).
    Since
    \[
    \bm x(\lambda)=(1-\lambda)\bm p+\lambda\bm q
    \]
    has nonnegative coefficient vector of Euclidean norm at least
    \(1/\sqrt2\), it follows that \(ch\le|\bm x(\lambda)|\le Ch\).

    The point \(\bm z(\lambda)\) belongs to \(T\cap\Gamma\).  Hence
    \[
    |\bm z(\lambda)|\ge |d(\bm o_0)|\ge c_0h,
    \qquad
    |\bm z(\lambda)|\le h.
    \]
    Since \(\bm z=\alpha\bm x\), these estimates yield
    \(c\le\alpha\le C\).  Moreover,
    \[
    \det[\partial_\lambda\bm z,\bm z]
    =
    \alpha^2\det[\partial_\lambda\bm x,\bm x]
    =
    -\alpha^2\det[\bm p,\bm q],
    \]
    and therefore
    \[
    |\partial_\lambda\bm z|
    \ge
    \frac{|\det[\partial_\lambda\bm z,\bm z]|}{|\bm z|}
    \ge ch.
    \]
    Finally, the differentiated level-set equation gives
    \[
    \partial_\lambda\alpha
    =
    -\frac{\alpha\,\partial_\lambda\bm x\cdot\nabla F(\bm z)}
    {\bm x\cdot\nabla F(\bm z)}.
    \]
    The numerator is \(O(h)\), while uniform transversality gives a
    denominator bounded below by \(ch\).  Thus
    \(\partial_\lambda\alpha=O(1)\), and
    \(\partial_\lambda\bm z
    =(\partial_\lambda\alpha)\bm x+\alpha\partial_\lambda\bm x=O(h)\).
    \end{proof}

    In particular, the absolute value in
    \(|\partial_\lambda\bm z|\) is applied to a vector that is uniformly
    nonzero after scaling by \(h\); all derivatives used below are therefore
    well defined.  To prove the higher-order estimates, we next need a lemma
    evaluating the inverse of the associated Jacobian matrix.

    \begin{lemma}
    \label{lem:matrix2d}
        Let
        \[
            M=\begin{pmatrix}
                \bm{\gamma}'(\sigma) & \dfrac{\bm{x}(\lambda)}{h}
            \end{pmatrix}.
        \]
        Then
        \[
            M^{-1}\partial_\lambda\bm{x}(\lambda)
            =
            \begin{pmatrix}
                O(h)\\
                O(h^2)
            \end{pmatrix},
        \]
        and $M^{-1}\bm{w}=O(|\bm{w}|)$ for every vector $\bm{w}\in\mathbb{R}^2$.
    \end{lemma}
    \begin{proof}[Proof of Lemma \ref{lem:matrix2d}]
        Since the vectors $\bm{\gamma}'(\sigma)$ and $\bm{x}(\lambda)/h$ have lengths of $O(1)$ and the consistency assumption prevents them from becoming parallel, $M^{-1}$ is uniformly bounded. It remains to prove the improved estimate for the vector $\partial_\lambda\bm{x}$.

        Let $\bm{\tau}=\partial_\lambda\bm{x}/|\partial_\lambda\bm{x}|$ whenever $\partial_\lambda\bm{x}\ne0$. The approximating segment in Section \ref{sec:2dcurve} has length $O(h)$ and its endpoints lie on the two cut edges associated with the same curve segment. Hence the standard chord-tangent estimate for a smooth curve gives that the angle between $\bm{\tau}$ and the tangent direction $\bm{\gamma}'(\sigma)$ is $O(h)$. Equivalently, if $d$ is the signed distance function to $\Gamma$ and $\bm{n}=\nabla d$ is the unit normal, then
        \[
            |\bm{n}\cdot \bm{\tau}|=O(h).
        \]
        Thus the component of $\bm{\tau}$ in the second column direction $\bm{x}/h$ is $O(h)$ after resolving $\bm{\tau}$ in the basis $\{\bm{\gamma}'(\sigma),\bm{x}/h\}$. Therefore
        \[
            M^{-1}\bm{\tau}
            =
            \begin{pmatrix}
                O(1)\\
                O(h)
            \end{pmatrix}.
        \]
        Since $|\partial_\lambda\bm{x}|=O(h)$, we obtain
        \[
            M^{-1}\partial_\lambda\bm{x}
            =
            \begin{pmatrix}
                O(h)\\
                O(h^2)
            \end{pmatrix}.
        \]
    \end{proof}
\begin{theorem}
\label{thm:2dbound}
    For any integer $n \in \mathbb{N}_+$, the parametrization
    $\bm{z}(\lambda)=\alpha(\lambda)\bm{x}(\lambda)$ satisfies
    \[
        \left|\partial_\lambda^n\bm{z}(\lambda)\right|
        \le
        C h^{\left\lfloor\frac{n}{2}\right\rfloor+1},
        \qquad \lambda\in[0,1],
    \]
    where \(C\) may depend on the smoothness and geometry of \(\Gamma\), the
    mesh shape-regularity constant, and the uniform transversality constant,
    but is independent of \(h\).
\end{theorem}
\begin{proof}
    We prove Theorem \ref{thm:2dbound} by induction. Taking the derivative of \eqref{eq:restriction2d} with respect to $\lambda$ gives
    \begin{equation}
    \label{eq:first_deriv_2d_new}
        M
        \begin{pmatrix}
            \sigma'(\lambda)\\[2pt]
            -h\rho'(\lambda)
        \end{pmatrix}
        =
        (1+\rho(\lambda))\partial_\lambda\bm{x}(\lambda).
    \end{equation}
    By Lemma \ref{lem:matrix2d}, this gives
    \[
        \sigma'(\lambda)=O(h),\qquad
        \rho'(\lambda)=O(h).
    \]
    Since $\partial_\lambda\bm{z}=\bm{\gamma}'(\sigma)\sigma'$, we have
    $\partial_\lambda\bm{z}=O(h)$, which proves the base case.

    For the induction step, we prove simultaneously that
    \[
        \partial_\lambda^n\sigma
        =
        O\!\left(h^{\left\lfloor\frac{n}{2}\right\rfloor+1}\right),
        \qquad
        \partial_\lambda^n\rho
        =
        O\!\left(h^{\left\lfloor\frac{n+1}{2}\right\rfloor}\right).
    \]
    Taking the $n$-th derivative of \eqref{eq:restriction2d} and applying Fa\`a di Bruno's formula yields
    \[
    M
    \begin{pmatrix}
        \partial_\lambda^n\sigma\\[4pt]
        -h\partial_\lambda^n\rho
    \end{pmatrix}
    =
    n(\partial_\lambda^{n-1}\rho)\partial_\lambda\bm{x}
    -
    \sum_{m=2}^n
    \sum_{\sum i_j=n}
    C_{i_1\dots i_m}
    (\partial_\lambda^{i_1}\sigma)\cdots(\partial_\lambda^{i_m}\sigma)
    \bm{\gamma}^{(m)}(\sigma).
    \]
    By the induction hypothesis, the first term on the right-hand side is
    \[
        n(\partial_\lambda^{n-1}\rho)\partial_\lambda\bm{x}
        =
        O\!\left(h^{\left\lfloor\frac{n}{2}\right\rfloor+1}\right)
    \]
    and, more precisely, it is a scalar multiple of $\partial_\lambda\bm{x}$. Therefore Lemma \ref{lem:matrix2d} gives the contribution
    \[
        \begin{pmatrix}
        O\!\left(h^{\left\lfloor\frac{n}{2}\right\rfloor+1}\right)\\[2pt]
        O\!\left(h^{\left\lfloor\frac{n}{2}\right\rfloor+2}\right)
        \end{pmatrix}
    \]
    to $(\partial_\lambda^n\sigma,-h\partial_\lambda^n\rho)^T$.

    For the summation term, applying the induction hypothesis to the derivatives of $\sigma$ and using Lemma \ref{lem:algebraic}, the product of the lower-order terms is bounded by
    \[
        O\!\left(h^{\left\lfloor\frac{n+1}{2}\right\rfloor+1}\right).
    \]
    Since $M^{-1}$ is uniformly bounded for arbitrary vectors, this term contributes
    \[
        O\!\left(h^{\left\lfloor\frac{n+1}{2}\right\rfloor+1}\right)
    \]
    to both components of $(\partial_\lambda^n\sigma,-h\partial_\lambda^n\rho)^T$. Combining the two contributions yields
    \[
        \partial_\lambda^n\sigma
        =
        O\!\left(h^{\left\lfloor\frac{n}{2}\right\rfloor+1}\right),
        \qquad
        \partial_\lambda^n\rho
        =
        O\!\left(h^{\left\lfloor\frac{n+1}{2}\right\rfloor}\right).
    \]
    Finally, applying Fa\`a di Bruno's formula once more to
    $\bm{z}(\lambda)=\bm{\gamma}(\sigma(\lambda))$ gives
    \[
        \partial_\lambda^n\bm{z}(\lambda)
        =
        O\!\left(h^{\left\lfloor\frac{n}{2}\right\rfloor+1}\right).
    \]
    This completes the induction.
\end{proof}

By Theorem \ref{thm:2dbound}, the highest derivative term required for the local quadrature error satisfies
\[
\partial_\lambda^{2N_q+1}\bm{z}(\lambda)=O(h^{N_q+1}).
\]
Therefore,
\[
\frac{d^{2N_q}}{d\lambda^{2N_q}}
\bigl(f(\bm{z}(\lambda))|\partial_\lambda\bm{z}(\lambda)|\bigr)
=
O(h^{N_q+1}).
\]
Hence the local error over one cut triangle is $O(h^{N_q+1})$. Since a smooth curve intersects $O(h^{-1})$ triangles, the global quadrature error for the curve integral is
\[
    e_{curve}=O(h^{N_q}).
\]

\subsection{Error over Surfaces}

We now analyze the surface parametrization introduced in
Section~\ref{sec:3dplane}. Let $T$ be a shape-regular tetrahedron consistent
with $\Gamma$, and let $h=h_T$. After translation, assume that the apex vertex
used in the ray parametrization is $\bm{o}_0=0$.

Let
\[
    \widehat K
    =
    \left\{
    \bm\lambda=(\lambda_1,\lambda_2)\in\mathbb{R}^2:
    \lambda_1\ge0,\ \lambda_2\ge0,\ \lambda_1+\lambda_2\le1
    \right\}.
\]
The chord triangle is parametrized by the affine map
\begin{equation}
\label{eq:x-affine-triangle-52}
    \bm{x}(\bm\lambda)
    =
    (1-\lambda_1-\lambda_2)\bm p
    +\lambda_1\bm q
    +\lambda_2\bm r ,
    \qquad
    \bm\lambda\in\widehat K .
\end{equation}
The exact surface patch is parametrized by
\[
    \bm z(\bm\lambda)
    =
    \alpha(\bm\lambda)\bm x(\bm\lambda),
    \qquad
    F(\bm z(\bm\lambda))=0 .
\]
Because $\bm x$ is affine, we have
\begin{equation}
\label{eq:x-affine-derivatives}
    |\bm x|+|\partial_{\lambda_1}\bm x|
    +|\partial_{\lambda_2}\bm x|
    \le Ch,
    \qquad
    \partial_{\bm\lambda}^{\nu}\bm x=0
    \quad\text{for all }|\nu|\ge2 .
\end{equation}

We use the Duffy-transformed tensor-product Gauss--Legendre rule described in
Section~\ref{sec:3dplane}. For
\[
\widetilde g(\xi,\eta)
=
(1-\xi)g\bigl(\xi,(1-\xi)\eta\bigr),
\]
the standard tensor-product Gauss--Legendre remainder gives
\begin{equation}
\label{eq:triangle-quadrature-error-assumption}
    \left|
    \int_{\widehat K} g(\bm\lambda)\,d\bm\lambda
    -
    Q_{\widehat K}^{N_q}(g)
    \right|
    \le
    C_{N_q}
    \left(
    \|\partial_\xi^{2N_q}\widetilde g\|_{L^\infty([0,1]^2)}
    +
    \|\partial_\eta^{2N_q}\widetilde g\|_{L^\infty([0,1]^2)}
    \right).
\end{equation}
The constant \(C_{N_q}\) depends only on the one-dimensional reference rule,
not on \(h\).

Let $\bm\gamma(\sigma_1,\sigma_2)$ be a smooth local parametrization of
$\Gamma$ on the patch $T\cap\Gamma$, with uniformly bounded derivatives, such
that
\[
    \bm z(\bm\lambda)
    =
    \bm\gamma(\bm\sigma(\bm\lambda)),
    \qquad
    \bm\sigma=(\sigma_1,\sigma_2).
\]
Set
\[
    \rho(\bm\lambda)=\alpha(\bm\lambda)-1 .
\]
Then
\begin{equation}
\label{eq:restriction3d-triangle}
    \bm\gamma(\bm\sigma(\bm\lambda))
    -
    (1+\rho(\bm\lambda))\bm x(\bm\lambda)
    =
    0 .
\end{equation}
The usual chord-surface estimate for a smooth surface gives
\[
    \rho(\bm\lambda)=O(h),
    \qquad
    \bm\lambda\in\widehat K .
\]

We first record the basic matrix estimate. Define
\[
    M
    =
    \begin{pmatrix}
        \bm\gamma_{\sigma_1} &
        \bm\gamma_{\sigma_2} &
        \dfrac{\bm x(\bm\lambda)}{h}
    \end{pmatrix}.
\]

\begin{lemma}
\label{lem:matrix3d}
For $i=1,2$,
\begin{equation}
\label{eq:matrix-improved-estimate-3d}
    M^{-1}\partial_{\lambda_i}\bm x
    =
    \begin{pmatrix}
        O(h)\\
        O(h)\\
        O(h^2)
    \end{pmatrix}.
\end{equation}
Moreover,
\[
    M^{-1}\bm w=O(|\bm w|)
    \qquad
    \text{for every }\bm w\in\mathbb{R}^3 .
\]
\end{lemma}

\begin{proof}[Proof of Lemma~\ref{lem:matrix3d}]
The vectors $\bm\gamma_{\sigma_1}$ and $\bm\gamma_{\sigma_2}$ span the tangent
plane of $\Gamma$. By the uniform transversality estimate
\eqref{eq:uniform-transversality}, the vector $\bm x/h$ is uniformly
transversal to this tangent plane. Hence $M$ is uniformly invertible and
$M^{-1}\bm w=O(|\bm w|)$.

It remains to prove the improved estimate
\eqref{eq:matrix-improved-estimate-3d}. Since $\bm p,\bm q,\bm r$ lie on the
smooth surface $\Gamma$ and have mutual distance $O(h)$, the chord-plane
estimate gives
\[
    \bm n(\bm z)\cdot(\bm q-\bm p)=O(h^2),
    \qquad
    \bm n(\bm z)\cdot(\bm r-\bm p)=O(h^2),
\]
uniformly for $\bm z\in T\cap\Gamma$. Therefore the normal components of
\[
    \partial_{\lambda_1}\bm x=\bm q-\bm p,
    \qquad
    \partial_{\lambda_2}\bm x=\bm r-\bm p
\]
are $O(h^2)$, while their tangential components are $O(h)$. Resolving
$\partial_{\lambda_i}\bm x$ in the basis
\[
    \left\{
    \bm\gamma_{\sigma_1},
    \bm\gamma_{\sigma_2},
    \bm x/h
    \right\}
\]
therefore gives
\[
    M^{-1}\partial_{\lambda_i}\bm x
    =
    \begin{pmatrix}
        O(h)\\
        O(h)\\
        O(h^2)
    \end{pmatrix},
    \qquad i=1,2.
\]
\end{proof}
\begin{theorem}
\label{thm:3dbound}
For any multi-index $\nu=(\nu_1,\nu_2)$ with $|\nu|=n\ge1$, the parametrization
\[
    \bm z(\bm\lambda)=\alpha(\bm\lambda)\bm x(\bm\lambda)
\]
satisfies
\begin{equation}
\label{eq:z-derivative-bound-3d}
    \left|
    \partial_{\bm\lambda}^{\nu}\bm z(\bm\lambda)
    \right|
    \le
    C h^{\left\lfloor n/2\right\rfloor+1},
    \qquad
    \bm\lambda\in\widehat K ,
\end{equation}
where \(C\) may depend on the smoothness and geometry of \(\Gamma\), the mesh
shape-regularity constant, and the uniform transversality constant, but is
independent of \(h\).
\end{theorem}

\begin{proof}
We prove the theorem by induction. We shall show that for every multi-index
$\nu$ with $|\nu|=n\ge1$,
\begin{equation}
\label{eq:sigma-rho-induction-3d}
    \partial_{\bm\lambda}^{\nu}\sigma_1,\,
    \partial_{\bm\lambda}^{\nu}\sigma_2
    =
    O\!\left(h^{\left\lfloor n/2\right\rfloor+1}\right),
    \qquad
    \partial_{\bm\lambda}^{\nu}\rho
    =
    O\!\left(h^{\left\lfloor (n+1)/2\right\rfloor}\right).
\end{equation}

For $n=1$, applying $\partial_{\lambda_i}$ to
\eqref{eq:restriction3d-triangle} gives
\[
    M
    \begin{pmatrix}
        \partial_{\lambda_i}\sigma_1\\[2pt]
        \partial_{\lambda_i}\sigma_2\\[2pt]
        -h\partial_{\lambda_i}\rho
    \end{pmatrix}
    =
    (1+\rho)\partial_{\lambda_i}\bm x,
    \qquad i=1,2 .
\]
Lemma~\ref{lem:matrix3d} implies
\[
    \partial_{\lambda_i}\sigma_1,\,
    \partial_{\lambda_i}\sigma_2
    =
    O(h),
    \qquad
    \partial_{\lambda_i}\rho=O(h).
\]
Thus the induction starts.

Now assume that \eqref{eq:sigma-rho-induction-3d} holds for all multi-indices
of order less than $n$, and let $|\nu|=n\ge2$. Applying
$\partial_{\bm\lambda}^{\nu}$ to \eqref{eq:restriction3d-triangle} yields
\begin{equation}
\label{eq:induction-equation-3d}
    M
    \begin{pmatrix}
        \partial_{\bm\lambda}^{\nu}\sigma_1\\[3pt]
        \partial_{\bm\lambda}^{\nu}\sigma_2\\[3pt]
        -h\partial_{\bm\lambda}^{\nu}\rho
    \end{pmatrix}
    =
    \nu_1
    \bigl(\partial_{\bm\lambda}^{\nu-e_1}\rho\bigr)
    \partial_{\lambda_1}\bm x
    +
    \nu_2
    \bigl(\partial_{\bm\lambda}^{\nu-e_2}\rho\bigr)
    \partial_{\lambda_2}\bm x
    -
    \bm S_\nu ,
\end{equation}
where $e_1=(1,0)$, $e_2=(0,1)$, and terms with negative multi-index
components are omitted. The term $\bm S_\nu$ contains only lower-order
nonlinear chain-rule terms coming from differentiating
$\bm\gamma(\bm\sigma)$. It is a finite sum of terms of the form
\[
    C\,
    \partial_{\sigma_1}^{a}\partial_{\sigma_2}^{b}
    \bm\gamma(\bm\sigma)
    \prod_{j=1}^{m}
    \partial_{\bm\lambda}^{\beta_j}\sigma_{c_j},
    \qquad
    m\ge2,\quad
    \sum_{j=1}^{m}\beta_j=\nu,\quad
    1\le |\beta_j|<n .
\]
Using the induction hypothesis and the elementary parity inequality
\[
    \sum_{j=1}^{m}
    \left(
    \left\lfloor\frac{|\beta_j|}{2}\right\rfloor+1
    \right)
    \ge
    \left\lfloor\frac{n+1}{2}\right\rfloor+1,
\]
we obtain
\begin{equation}
\label{eq:Snu-bound-3d}
    \bm S_\nu
    =
    O\!\left(
    h^{\left\lfloor (n+1)/2\right\rfloor+1}
    \right).
\end{equation}

For the first two terms on the right-hand side of
\eqref{eq:induction-equation-3d}, the induction hypothesis gives
\[
    \partial_{\bm\lambda}^{\nu-e_i}\rho
    =
    O\!\left(h^{\left\lfloor n/2\right\rfloor}\right),
    \qquad i=1,2.
\]
Since these terms are scalar multiples of
$\partial_{\lambda_i}\bm x$, Lemma~\ref{lem:matrix3d} gives the contribution
\[
    \begin{pmatrix}
        O\!\left(h^{\left\lfloor n/2\right\rfloor+1}\right)\\[2pt]
        O\!\left(h^{\left\lfloor n/2\right\rfloor+1}\right)\\[2pt]
        O\!\left(h^{\left\lfloor n/2\right\rfloor+2}\right)
    \end{pmatrix}.
\]
For $\bm S_\nu$, the uniform boundedness of $M^{-1}$ and
\eqref{eq:Snu-bound-3d} imply the contribution
\[
    O\!\left(
    h^{\left\lfloor (n+1)/2\right\rfloor+1}
    \right)
\]
to all three components. Combining the estimates gives
\[
    \partial_{\bm\lambda}^{\nu}\sigma_1,\,
    \partial_{\bm\lambda}^{\nu}\sigma_2
    =
    O\!\left(h^{\left\lfloor n/2\right\rfloor+1}\right),
    \qquad
    \partial_{\bm\lambda}^{\nu}\rho
    =
    O\!\left(h^{\left\lfloor (n+1)/2\right\rfloor}\right).
\]
This proves \eqref{eq:sigma-rho-induction-3d}.

Finally, applying the multivariate Fa\`a di Bruno formula to
\[
    \bm z(\bm\lambda)=\bm\gamma(\bm\sigma(\bm\lambda))
\]
and using the just-proved estimates for $\bm\sigma$ gives
\[
    \partial_{\bm\lambda}^{\nu}\bm z(\bm\lambda)
    =
    O\!\left(h^{\left\lfloor n/2\right\rfloor+1}\right).
\]
The theorem follows.
\end{proof}

We next estimate the surface Jacobian. Define
\[
    A(\bm\lambda)
    =
    \left|
        \partial_{\lambda_1}\bm z(\bm\lambda)
        \times
        \partial_{\lambda_2}\bm z(\bm\lambda)
    \right|.
\]

\begin{lemma}
\label{lem:surface-jacobian-bound}
For any multi-index $\nu$ with $|\nu|=n\ge0$,
\begin{equation}
\label{eq:surface-jacobian-derivative-bound}
    \partial_{\bm\lambda}^{\nu}A(\bm\lambda)
    =
    O\!\left(h^{\left\lfloor (n+1)/2\right\rfloor+2}\right).
\end{equation}
\end{lemma}

\begin{proof}
Let
\[
    \bm B(\bm\lambda)
    =
    \partial_{\lambda_1}\bm z(\bm\lambda)
    \times
    \partial_{\lambda_2}\bm z(\bm\lambda).
\]
    We first justify the required lower bound.  With the apex translated to
    the origin, write
    \[
    \bm p=t_1\bm o_1,\qquad
    \bm q=t_2\bm o_2,\qquad
    \bm r=t_3\bm o_3 .
    \]
    Vertex--surface separation, the \(1\)-Lipschitz property of signed
    distance, and shape-regularity imply \(c\le t_i\le1\).  Therefore
    \[
    |\det[\bm p,\bm q,\bm r]|
    =
    t_1t_2t_3|\det[\bm o_1,\bm o_2,\bm o_3]|
    \ge ch^3 .
    \]
    In particular, the smallest singular value of
    \([\bm p,\bm q,\bm r]\) is comparable to \(h\).  Since
    \(\bm x=(1-\lambda_1-\lambda_2)\bm p+\lambda_1\bm q+\lambda_2\bm r\),
    this gives \(ch\le|\bm x|\le Ch\).  The same signed-distance argument
    gives \(ch\le|\bm z|\le h\), so
    \(c\le\alpha=|\bm z|/|\bm x|\le C\).

    Now use the scalar-triple-product identity
    \[
    \bm z\cdot
    \left(
    \partial_{\lambda_1}\bm z
    \times
    \partial_{\lambda_2}\bm z
    \right)
    =
    \alpha^3
    \bm x\cdot
    \left(
    \partial_{\lambda_1}\bm x
    \times
    \partial_{\lambda_2}\bm x
    \right)
    =
    \alpha^3\det[\bm p,\bm q,\bm r].
    \]
    Since \(|\bm z|\le h\), it follows that
    \[
    |\bm B(\bm\lambda)|
    \ge
    \frac{
    |\bm z\cdot\bm B(\bm\lambda)|
    }{|\bm z|}
    \ge ch^2 .
    \]
By the product rule,
\[
    \partial_{\bm\lambda}^{\nu}\bm B
    =
    \sum_{\beta+\eta=\nu}
    C_{\beta,\eta}
    \left(
        \partial_{\bm\lambda}^{\beta+e_1}\bm z
        \times
        \partial_{\bm\lambda}^{\eta+e_2}\bm z
    \right).
\]
Using Theorem~\ref{thm:3dbound}, each term is bounded by
\[
    O\!\left(
    h^{\left\lfloor (|\beta|+1)/2\right\rfloor+1}
    h^{\left\lfloor (|\eta|+1)/2\right\rfloor+1}
    \right)
    =
    O\!\left(
    h^{\left\lfloor (n+1)/2\right\rfloor+2}
    \right).
\]
Thus
\[
    \partial_{\bm\lambda}^{\nu}\bm B
    =
    O\!\left(h^{\left\lfloor (n+1)/2\right\rfloor+2}\right).
\]
Writing $\bm B=h^2\widetilde{\bm B}$, the vector
$\widetilde{\bm B}$ is uniformly bounded away from zero. Since
$\bm v\mapsto |\bm v|$ is smooth away from $\bm v=0$, repeated use of the
chain rule, together with the same elementary parity estimate for products of
derivatives of \(\widetilde{\bm B}\), gives the same scaling for
\(A=|\bm B|\):
\[
    \partial_{\bm\lambda}^{\nu}A
    =
    O\!\left(h^{\left\lfloor (n+1)/2\right\rfloor+2}\right).
\]
\end{proof}

Now set
\[
    G(\bm\lambda)
    =
    f(\bm z(\bm\lambda))A(\bm\lambda).
\]
The smoothness of $f$, Theorem~\ref{thm:3dbound}, and
Lemma~\ref{lem:surface-jacobian-bound} imply
\[
    \partial_{\bm\lambda}^{\nu}G(\bm\lambda)
    =
    O\!\left(h^{\left\lfloor (|\nu|+1)/2\right\rfloor+2}\right).
\]
In particular, derivatives of orders \(2N_q-1\) and \(2N_q\) both satisfy
\[
    \partial_{\bm\lambda}^{\nu}G(\bm\lambda)
    =
    O(h^{N_q+2}).
\]
For the Duffy pullback
\[
\widetilde G(\xi,\eta)
=(1-\xi)G\bigl(\xi,(1-\xi)\eta\bigr),
\]
the chain rule therefore gives
\[
\partial_\xi^{2N_q}\widetilde G
=O(h^{N_q+2}),
\qquad
\partial_\eta^{2N_q}\widetilde G
=O(h^{N_q+2}).
\]
Using \eqref{eq:triangle-quadrature-error-assumption}, the local quadrature
error over one cut tetrahedron satisfies
\[
    e_{\mathrm{surface},T}
    =
    O(h^{N_q+2}).
\]
Since a smooth surface intersects $O(h^{-2})$ tetrahedra, the global surface
quadrature error is
\[
    e_{\mathrm{surface}}
    =
    O(h^{N_q}).
\]

\subsection{Error over Regions}
\label{sec:regionorder}

We now analyze the region integral methods. The proof uses the boundary
parametrization estimates from Sections~\ref{sec:2dcurve} and
\ref{sec:3dplane}. The main point is that the additional radial variable
$\beta$ does not reduce the convergence order, because the Gauss--Legendre rule
in $\beta$ integrates the polynomial radial factors exactly up to high degree,
and all remaining $\beta$-derivatives of $f(\beta\bm z)$ bring additional powers
of $h$.

We first account for elements that are not cut by \(\Gamma\), as well as the
whole-element term used in a complement subtraction.  Let
\(Q_T^{N_q}\) be the affine pullback of a fixed positive simplex rule that is
exact for polynomials of total degree at most \(2N_q-1\), as specified in
Sections~\ref{sec:2dregion} and~\ref{sec:3dregion}.  The standard quadrature estimate on a shape-regular simplex in dimension
\(d\) gives
\begin{equation}
\left|
\int_T f-Q_T^{N_q}(f)
\right|
\le
C h_T^{2N_q+d}
|f|_{W^{2N_q,\infty}(T)} .
\label{eq:whole-simplex-error}
\end{equation}
There are \(O(h^{-d})\) whole elements, so their accumulated error is
\(O(h^{2N_q})\), and hence is \(O(h^{N_q})\) for \(N_q\ge1\).  In a
complement configuration, the numerical contribution is the difference
between \(Q_T^{N_q}(f)\) and the mapped cone quadrature.  The triangle
inequality shows that its error is bounded by the sum of the whole-element
error \eqref{eq:whole-simplex-error} and the corresponding cone error analyzed
below.

\subsubsection*{The two-dimensional region case}

Recall the parametrization in Section~\ref{sec:2dregion}:
\[
    \bm\phi(\lambda,\beta)
    =
    \beta\bm z(\lambda),
    \qquad
    (\lambda,\beta)\in[0,1]^2 .
\]
The Jacobian matrix is
\[
    J(\bm\phi)
    =
    \left[
        \beta\partial_\lambda\bm z(\lambda),
        \bm z(\lambda)
    \right].
\]
We write
\[
    |\det J(\bm\phi)|
    =
    \beta D_2(\lambda),
    \qquad
    D_2(\lambda)
    =
    \left|
    \det
    \left[
        \partial_\lambda\bm z(\lambda),
        \bm z(\lambda)
    \right]
    \right|.
\]
The use of the absolute value corresponds to the usual change-of-variables
formula. If the orientation is fixed in advance, the absolute value may be
removed.

\begin{lemma}
\label{lem:Db_bound}
For any integer $n\ge0$,
\[
    \partial_\lambda^n D_2(\lambda)
    =
    O\!\left(
    h^{\left\lfloor n/2\right\rfloor+2}
    \right).
\]
\end{lemma}

\begin{proof}
Let
\[
    \widetilde D_2(\lambda)
    =
    \det
    \left[
        \partial_\lambda\bm z(\lambda),
        \bm z(\lambda)
    \right].
\]
By the product rule,
\[
    \partial_\lambda^n \widetilde D_2(\lambda)
    =
    \sum_{k=0}^{n}
    \binom{n}{k}
    \det
    \left[
        \partial_\lambda^{k+1}\bm z(\lambda),
        \partial_\lambda^{n-k}\bm z(\lambda)
    \right].
\]
Using Theorem~\ref{thm:2dbound}, each term is bounded by
\[
    O\!\left(
    h^{\left\lfloor (k+1)/2\right\rfloor+1}
    h^{\left\lfloor (n-k)/2\right\rfloor+1}
    \right)
    =
    O\!\left(
    h^{\left\lfloor n/2\right\rfloor+2}
    \right).
\]
Moreover, the transversality condition gives
$|\widetilde D_2(\lambda)|\ge c h^2$. Hence
$D_2=|\widetilde D_2|$ is smooth with respect to $\lambda$ on the patch, and
the same estimate holds for $D_2$.
\end{proof}

Define
\[
    H_2(\lambda)
    =
    D_2(\lambda)
    \int_0^1
    \beta f(\beta\bm z(\lambda))\,d\beta .
\]
For $n\ge0$, the smoothness of $f$, Theorem~\ref{thm:2dbound}, and
Lemma~\ref{lem:Db_bound} imply
\[
    \partial_\lambda^n H_2(\lambda)
    =
    O\!\left(
    h^{\left\lfloor n/2\right\rfloor+2}
    \right).
\]
Therefore the $N_q$-point Gauss--Legendre rule in the $\lambda$ variable gives
the local error
\[
    O(h^{N_q+2}).
\]

It remains to check that the quadrature in the radial variable $\beta$ is not
worse. For fixed $\lambda$, consider
\[
    g_\lambda(\beta)=\beta f(\beta\bm z(\lambda)).
\]
Since the Gauss--Legendre rule with $N_q$ nodes is exact for polynomials of
degree at most $2N_q-1$, the standard one-dimensional remainder gives
\[
    \left|
    \int_0^1 g_\lambda(\beta)\,d\beta
    -
    Q_\beta^{N_q}(g_\lambda)
    \right|
    \le
    C
    \left\|
    \partial_\beta^{2N_q}g_\lambda
    \right\|_{L^\infty(0,1)} .
\]
Because each derivative of $f(\beta\bm z)$ with respect to $\beta$ produces a
factor $\bm z=O(h)$, and because the factor $\beta$ is a polynomial of degree
one, we have
\[
    \partial_\beta^{2N_q}g_\lambda
    =
    O(h^{2N_q-1}).
\]
Multiplying by $D_2(\lambda)=O(h^2)$ gives a radial quadrature contribution
\[
    O(h^{2N_q+1}),
\]
which is at least $O(h^{N_q+2})$ for $N_q\ge1$.

Consequently, the local quadrature error on one cut triangle is
\[
    e_{\mathrm{local},2D}
    =
    O(h^{N_q+2}).
\]
Because \(\Gamma\) is a smooth compact curve and the mesh is quasi-uniform,
only \(O(h^{-1})\) triangles are cut.  Thus the accumulated mapped-cone error
is \(O(h^{N_q+1})\).  Adding the \(O(h^{2N_q})\) whole-element error from
\eqref{eq:whole-simplex-error}, including the whole-element terms in
complement configurations, gives
\[
    E_{2D}
    =
    O(h^{N_q+1})+O(h^{2N_q})
    =
    O(h^{N_q}),
    \qquad N_q\ge1.
\]

\subsubsection*{The three-dimensional region case}

We now consider the three-dimensional region integral. With the triangular
surface parametrization introduced in Section~\ref{sec:3dplane}, the reference
domain is
\[
    \widehat K\times[0,1],
\]
where
\[
    \widehat K
    =
    \left\{
    (\lambda_1,\lambda_2):
    \lambda_1\ge0,\ \lambda_2\ge0,\ \lambda_1+\lambda_2\le1
    \right\}.
\]
The region parametrization is
\[
    \bm\phi(\lambda_1,\lambda_2,\beta)
    =
    \beta\bm z(\lambda_1,\lambda_2),
    \qquad
    (\lambda_1,\lambda_2)\in\widehat K,\quad
    0\le\beta\le1 .
\]
The Jacobian matrix is
\[
    J(\bm\phi)
    =
    \left[
        \beta\partial_{\lambda_1}\bm z,
        \beta\partial_{\lambda_2}\bm z,
        \bm z
    \right].
\]
Hence
\[
    |\det J(\bm\phi)|
    =
    \beta^2 D_3(\bm\lambda),
\]
where
\[
    D_3(\bm\lambda)
    =
    \left|
    \bm z(\bm\lambda)
    \cdot
    \left(
        \partial_{\lambda_1}\bm z(\bm\lambda)
        \times
        \partial_{\lambda_2}\bm z(\bm\lambda)
    \right)
    \right|.
\]

\begin{lemma}
\label{lem:Dlambda3d_bound}
For any multi-index $\nu$ with $|\nu|=n\ge0$,
\[
    \partial_{\bm\lambda}^{\nu}D_3(\bm\lambda)
    =
    O\!\left(
    h^{\left\lfloor (n+1)/2\right\rfloor+3}
    \right).
\]
\end{lemma}

\begin{proof}
Using \(\bm z=\alpha\bm x\) and the fact that \(\bm x\) is affine, the scalar
triple product simplifies to
\[
\bm z\cdot
\left(
\partial_{\lambda_1}\bm z
\times
\partial_{\lambda_2}\bm z
\right)
=
\alpha^3\,
\bm x\cdot
\left(
\partial_{\lambda_1}\bm x
\times
\partial_{\lambda_2}\bm x
\right).
\]
The second factor is constant on \(\widehat K\), has magnitude comparable to
\(h^3\), and has a fixed sign. Consequently,
\[
D_3(\bm\lambda)
=
\alpha(\bm\lambda)^3
\left|
\bm x\cdot
\left(
\partial_{\lambda_1}\bm x
\times
\partial_{\lambda_2}\bm x
\right)
\right|.
\]
The derivative estimates for
\(\alpha=1+\rho\) in \eqref{eq:sigma-rho-induction-3d}, followed by the product
rule, give the stated bound.
\end{proof}

Define
\[
    H_3(\bm\lambda)
    =
    D_3(\bm\lambda)
    \int_0^1
    \beta^2 f(\beta\bm z(\bm\lambda))\,d\beta .
\]
By Lemma~\ref{lem:Dlambda3d_bound}, Theorem~\ref{thm:3dbound}, and the
smoothness of $f$,
\[
    \partial_{\bm\lambda}^{\nu}H_3(\bm\lambda)
    =
    O\!\left(
    h^{\left\lfloor (|\nu|+1)/2\right\rfloor+3}
    \right).
\]
Thus derivatives of orders \(2N_q-1\) and \(2N_q\) are both
\(O(h^{N_q+3})\). For the Duffy pullback
\[
\widetilde H_3(\xi,\eta)
=(1-\xi)H_3\bigl(\xi,(1-\xi)\eta\bigr),
\]
the chain rule gives
\[
\partial_\xi^{2N_q}\widetilde H_3
=O(h^{N_q+3}),
\qquad
\partial_\eta^{2N_q}\widetilde H_3
=O(h^{N_q+3}).
\]
Using \eqref{eq:triangle-quadrature-error-assumption}, we obtain the local error
\[
    O(h^{N_q+3})
\]
from the $\bm\lambda$ quadrature.

For the radial direction, fix $\bm\lambda$ and set
\[
    g_{\bm\lambda}(\beta)
    =
    \beta^2 f(\beta\bm z(\bm\lambda)).
\]
For $N_q\ge2$, the $N_q$-point Gauss--Legendre rule is exact for polynomials up
to degree $2N_q-1$. Since the factor $\beta^2$ is polynomial and every
remaining derivative of $f(\beta\bm z)$ with respect to $\beta$ produces a
factor $\bm z=O(h)$, we have
\[
    \partial_\beta^{2N_q}g_{\bm\lambda}
    =
    O(h^{2N_q-2}).
\]
Multiplying by $D_3(\bm\lambda)=O(h^3)$ gives the radial quadrature
contribution
\[
    O(h^{2N_q+1}),
\]
which is $O(h^{N_q+3})$ for $N_q\ge2$.

For \(N_q=1\), the one-point rule is not exact for the factor \(\beta^2\).
Nevertheless, \(g_{\bm\lambda}\) is uniformly bounded and
\(D_3=O(h^3)\), so the radial error on one cut tetrahedron is \(O(h^3)\).
The tangential error is \(O(h^4)\) and is therefore smaller.  Consequently,
the local quadrature error on one cut tetrahedron is
\[
    e_{\mathrm{local},3D}
    =
    \begin{cases}
    O(h^3),&N_q=1,\\
    O(h^{N_q+3}),&N_q\ge2.
    \end{cases}
\]
Because a smooth compact surface cuts only \(O(h^{-2})\) tetrahedra, the
accumulated cone error is \(O(h)\) for \(N_q=1\) and
\(O(h^{N_q+1})\) for \(N_q\ge2\).  The whole-element contribution from
\eqref{eq:whole-simplex-error} is \(O(h^{2N_q})\).  Complement
configurations introduce the sum, rather than a new type, of these two
errors.  Hence
\[
    E_{3D}
    =
    O(h^{N_q}),
    \qquad N_q\ge1.
\]

Combining the two-dimensional and three-dimensional arguments, the region
quadrature preserves the same global order as the corresponding boundary
quadrature:
\[
    E_{\Omega}
    =
    O(h^{N_q}).
\]

\begin{theorem}[Global quadrature error]
\label{thm:global-quadrature-error}
Let \(d\in\{2,3\}\), let \(\Gamma\) and \(f\) satisfy the smoothness
assumptions stated in the Introduction, and let the adjusted meshes be
shape-regular, quasi-uniform, and uniformly consistent with \(\Gamma\).
Assume that the scalar intersection equations are solved to machine precision.
Denote by \(Q_{\Gamma,h}^{N_q}(f)\) and \(Q_{\Omega,h}^{N_q}(f)\) the assembled
boundary and region rules described above.  Then there is a constant \(C\),
independent of \(h\), such that, above the machine-precision floor,
\[
\left|
\int_\Gamma f\,dS-Q_{\Gamma,h}^{N_q}(f)
\right|
\le Ch^{N_q},
\qquad
\left|
\int_\Omega f\,dx-Q_{\Omega,h}^{N_q}(f)
\right|
\le Ch^{N_q}.
\]
\end{theorem}

\begin{proof}
For \(d=2\), Theorem~\ref{thm:2dbound} and
Lemma~\ref{lem:curve-speed-lower-bound} give a local boundary error
\(O(h^{N_q+1})\); summing over the \(O(h^{-1})\) cut triangles yields
\(O(h^{N_q})\).  For \(d=3\), Theorem~\ref{thm:3dbound} and
Lemma~\ref{lem:surface-jacobian-bound} give a local boundary error
\(O(h^{N_q+2})\); summing over the \(O(h^{-2})\) cut tetrahedra again yields
\(O(h^{N_q})\).  The region estimates are precisely those proved in
Section~\ref{sec:regionorder}, including the \(N_q=1\) radial case, all uncut
elements, and complement subtractions.
\end{proof}

\section{Numerical Experiments}
\label{sec:surfacetest}
We present numerical experiments for curve and surface integrals over implicitly
defined geometries. In all tests, the proposed quadrature rule is applied on a
background simplicial mesh, and the error is measured against either an analytic
value or a high-accuracy reference value. The parameter \(q\) denotes the number of
Gauss--Legendre nodes in each parameter direction.

\textbf{Example 1: Integration over quartic squircle.}
We first consider a smooth non-conic closed curve.
Let \(U\subset\mathbb R^2\) be a computational box containing the curve, and
let
\[
    \Gamma=\{(x,y)\in U:\ F(x,y)=0\}, \quad  F(x,y)
    =
    \left(\frac{x}{A}\right)^4
    +
    \left(\frac{y}{B}\right)^4
    +
    \alpha
    \left(\frac{x}{A}\right)^2
    \left(\frac{y}{B}\right)^2
    -1 
\]
with $A=0.80, \; B=0.62,\; \alpha=0.22.$

The integral we compute is
\[
    I=\int_{\Gamma} f(x,y)\,ds, \quad f(x,y)=\exp(0.20x-0.15y)+0.50x^2+y^2 .
\]
A high-accuracy reference computation gives
   $I_{\rm ref}\approx 6.954045469673768.$
Uniform background grids with
$
    n=8,16,24,32
$
are used, together with Gauss--Legendre orders
$
    q=2,4,6,8 .
$

\begin{figure}[!htbp]
  \centering
 
      \includegraphics[width=0.45\linewidth]{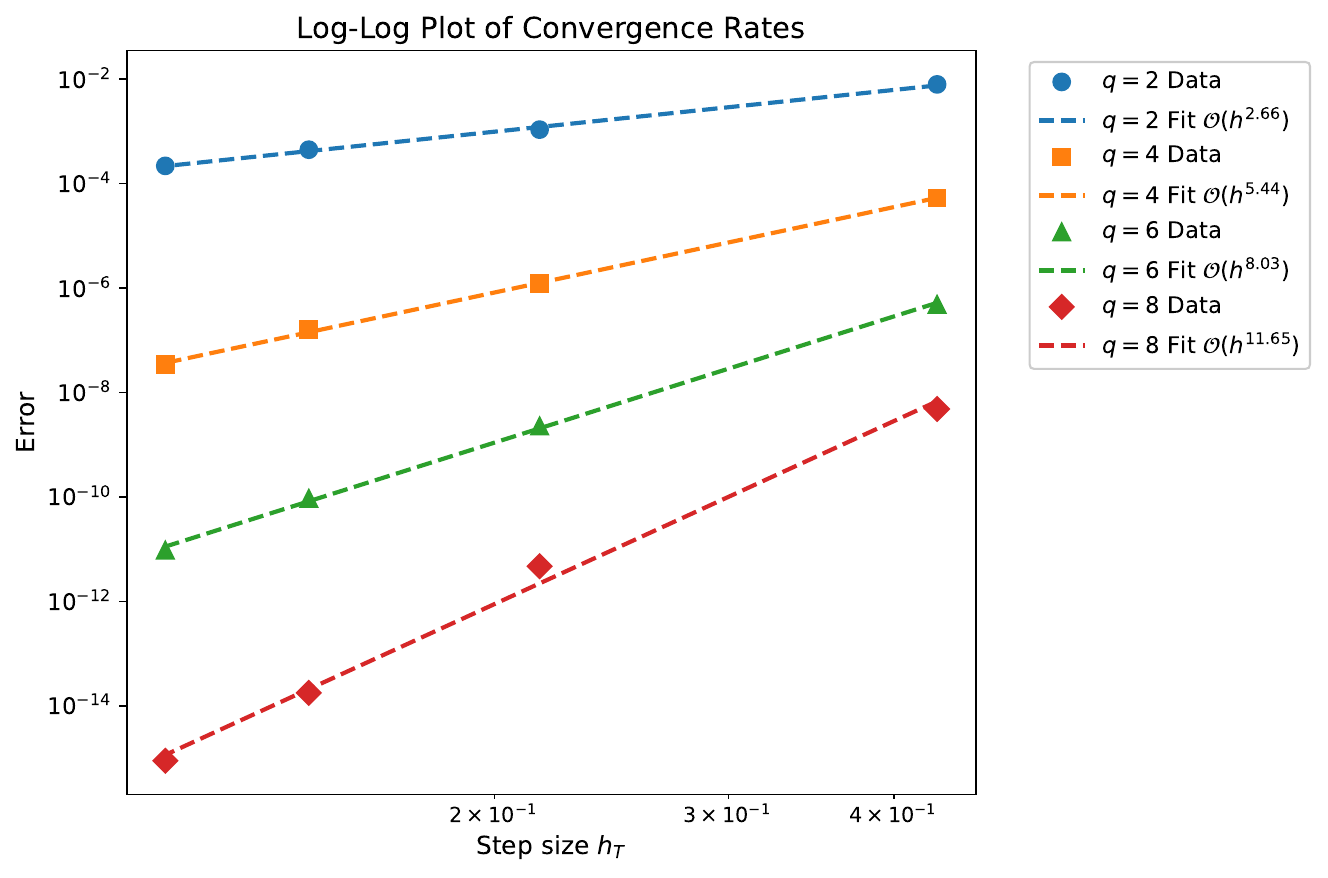} 
           \includegraphics[width=0.45\linewidth]{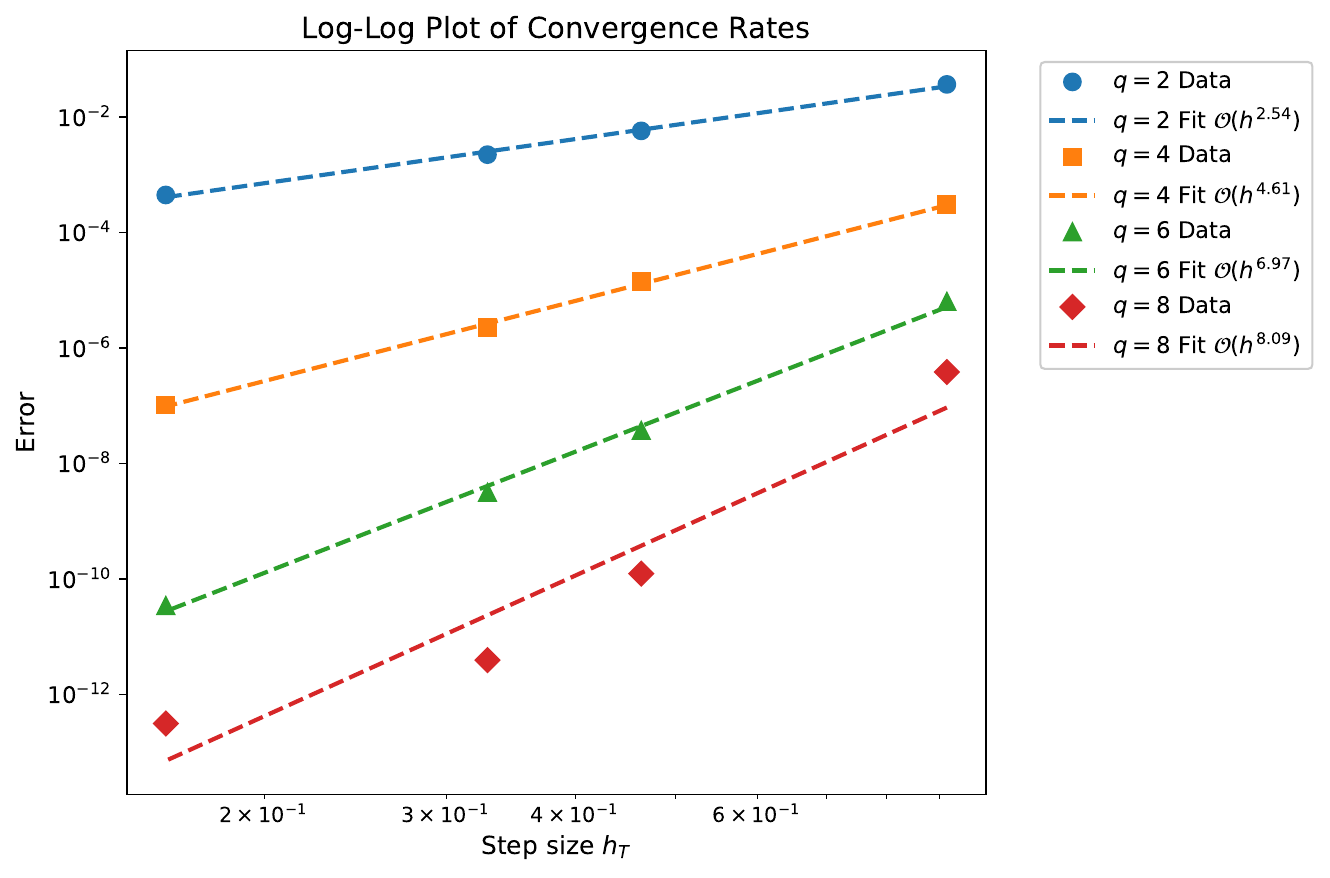} 
      \caption{Error in Example 1 (left) and 2 (right).} 
      \label{fig:test12}
 
\end{figure}

The errors are shown in Figure~\ref{fig:test12}. The observed convergence rate
increases with \(q\), which confirms the high-order behavior of the method.\\





 

\textbf{Example 2: Ellipsoid} Consider the surface $\Gamma \subset \mathbb{R}^3$ given by $x^2+4y^2+9z^2=1$. The integrand is chosen as $f=1$ so that the result is the surface area. The error with different $n$ and $q$ are shown in Fig. \ref{fig:test2.1}.\\

\begin{figure}[!htbp]
  \centering
 
      \includegraphics[width=0.45\linewidth]{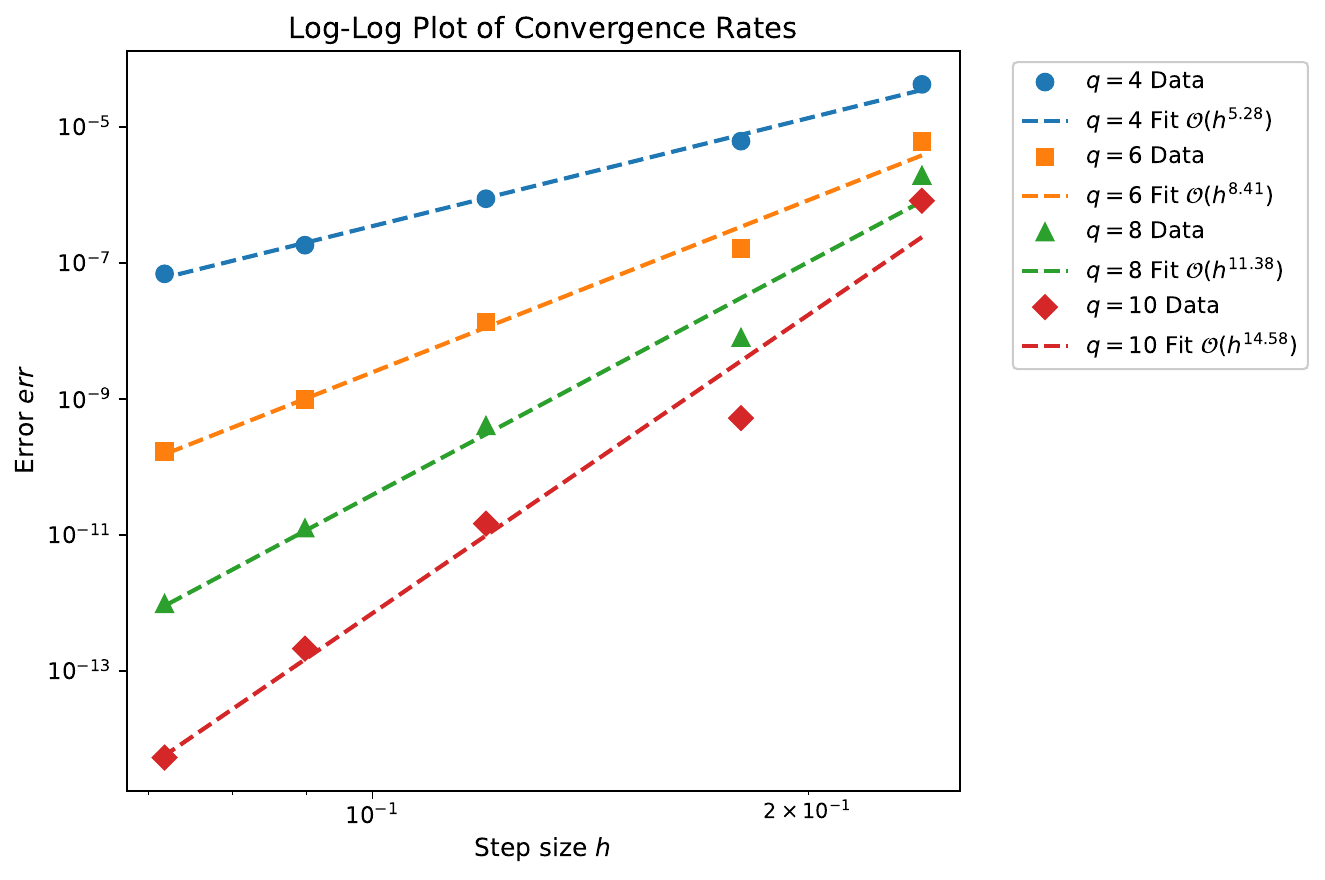} 
       \includegraphics[width=0.45\linewidth]{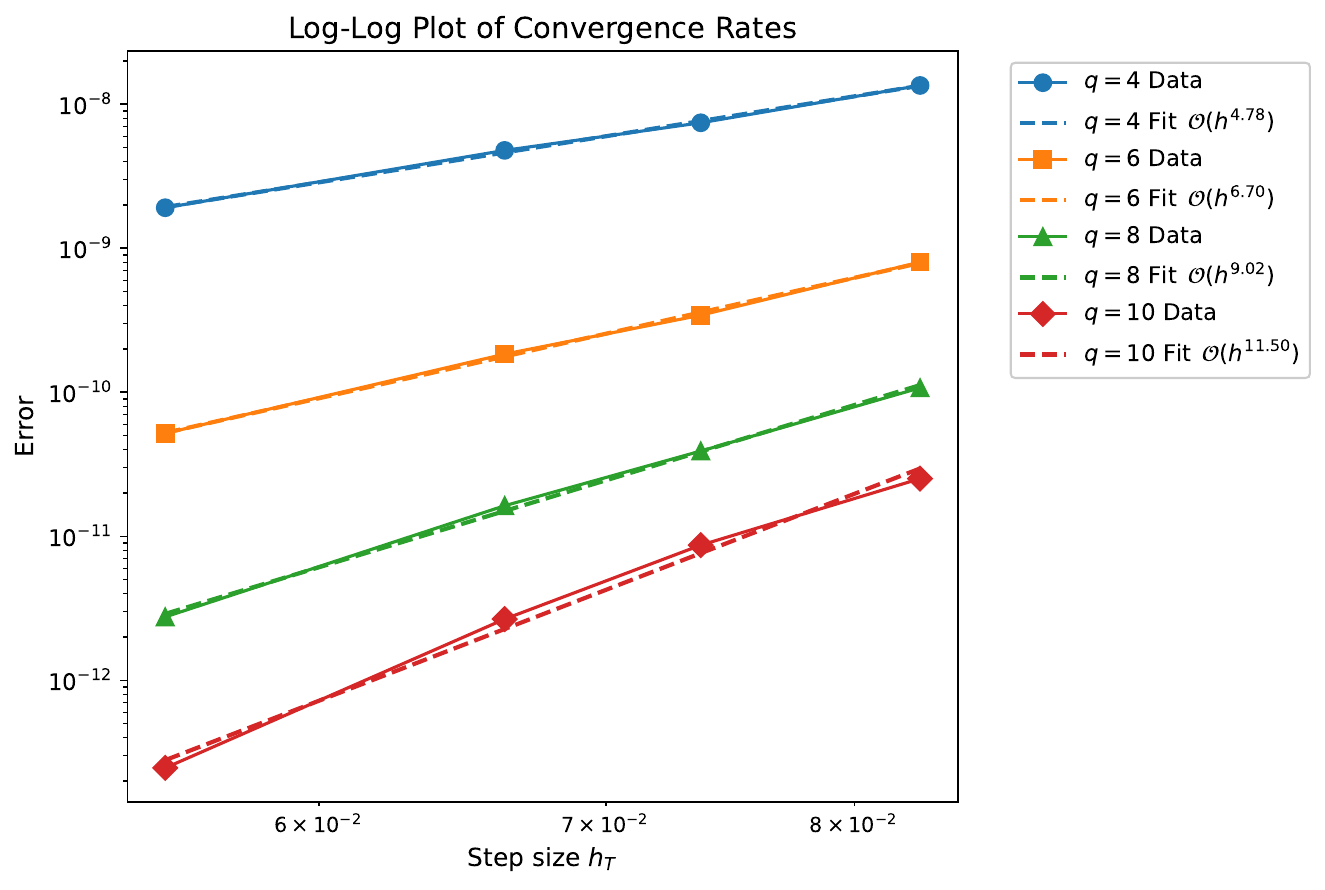} 
      \caption{Error in Example 2. Left: surface integral; Right: region integral.} 
      \label{fig:test2.1}
 
\end{figure}

For the region integral, we choose the integrand
$f(x,y,z)=1+x^2+2y^2+3z^2$ in the ellipsoid
$x^2+4y^2+9z^2\le1$. We compute the integral with $q=4,6,8,10$ and plot the error in Fig. \ref{fig:test2.1}.\\
 
 

\textbf{Test 3: Torus}
A torus $\Gamma = \{ F(x,y,z) = 0 \}$ is given by $F = (x^2 + y^2 + z^2 + R^2 - r^2)^2 - 4R^2(x^2 + y^2)$. We numerically compute the torus surface area $\int_{\Gamma} 1 \, \mathrm{d}S$ using this method and its analytical true value is $4\pi^2 R r$.We choose radius $R=0.8, r=0.35$ and $n = 12,\, 16,\, 24,\, 32,\, 40$,  with Gauss node degrees $q = 4,\, 6,\, 8,\, 10$. The results are shown in Fig. \ref{fig:test3.1}.
 \begin{figure}[!htbp]
   \centering
 
       \includegraphics[width=0.45\linewidth]{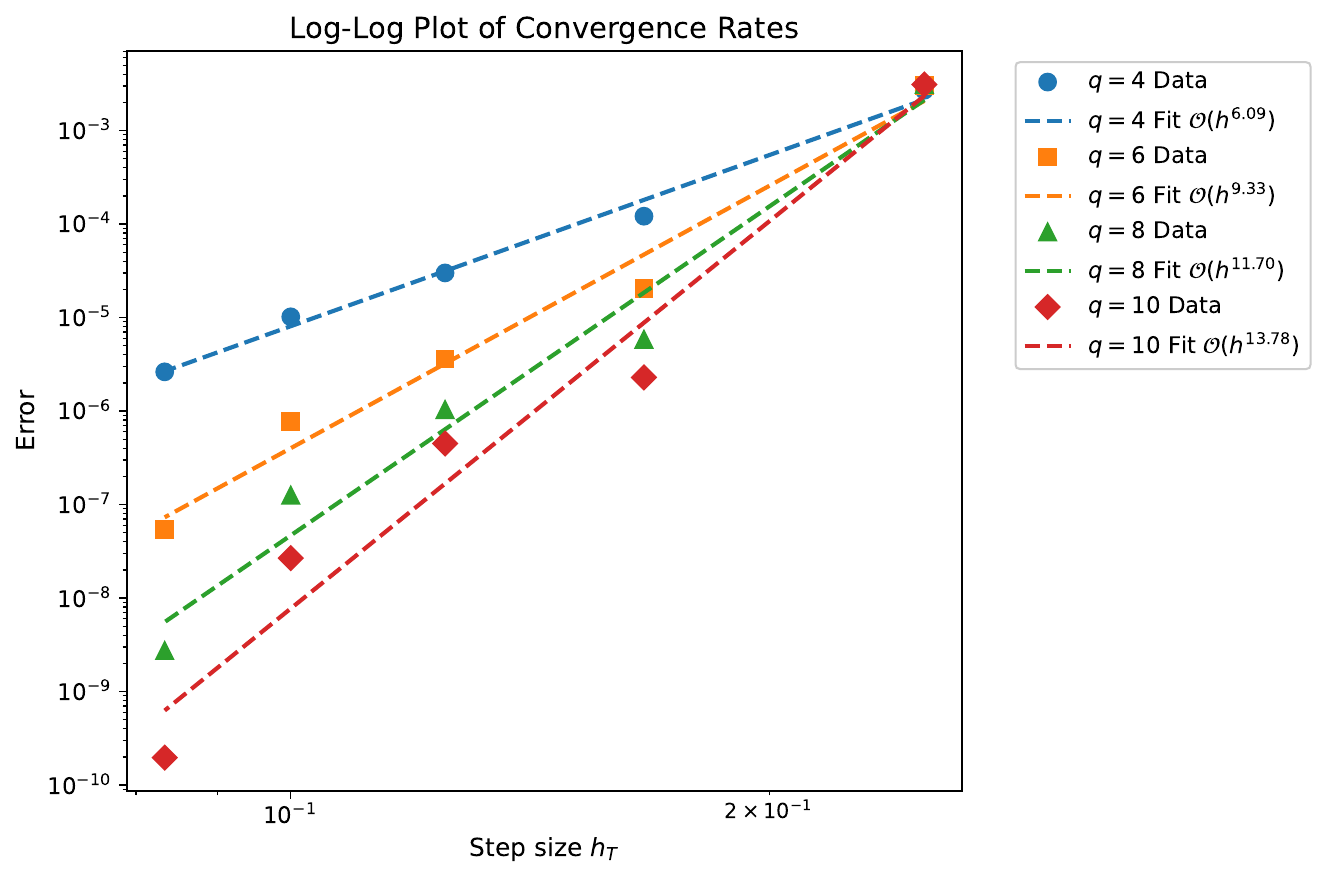} 
       \includegraphics[width=0.45\linewidth]{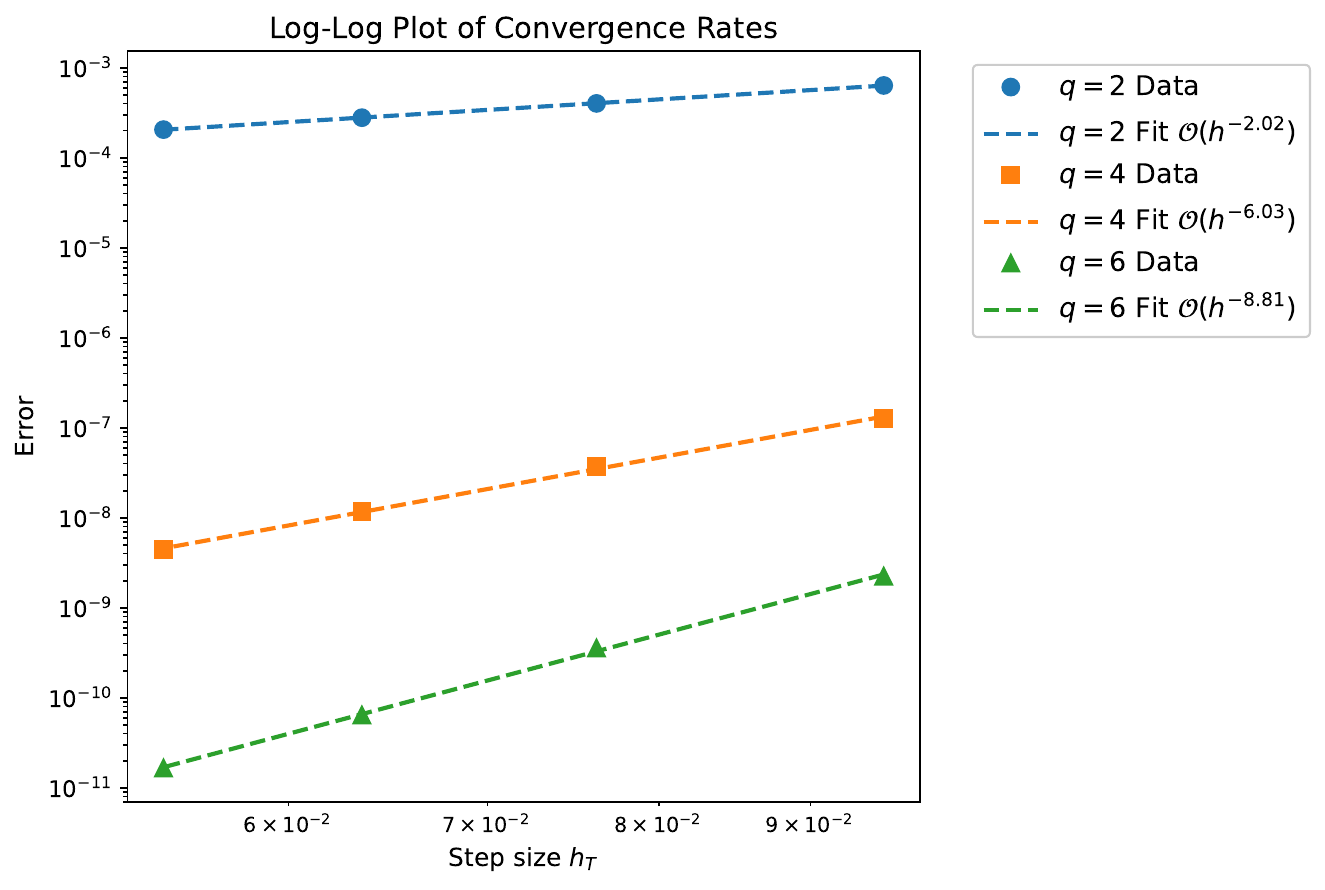} 
     \caption{Error in Example 3. Left: surface integral; Right: region integral} 
       \label{fig:test3.1}
 
\end{figure}

For the region case, we set the integrand to be $f(x,y,z) = 1 + x^2 + 2y^2 + 3z^2$ in the region $ (x^2 + y^2 + z^2 + R^2 - r^2)^2 - 4R^2(x^2 + y^2)\le 0$. We choose radius $R=1.5, r=0.5$ and Gauss node degrees $q=2,4,6$ to calculate the region integral, see Fig. \ref{fig:test3.1}.\\
\textbf{Example 5: A Enzensberger--Stern-type algebraic surface.}
To examine the performance of the proposed quadrature method on a geometrically
nontrivial closed surface, we consider a smoothed variant of the
Enzensberger--Stern algebraic surface.
  

Let
\[
    \Gamma_{a,b}
    =
    \left\{
    (x,y,z)\in\mathbb{R}^{3}:
    \phi_{a,b}(x,y,z)=0
    \right\},
\]
where
\[
\begin{aligned}
    \phi_{a,b}(x,y,z)
    =
    a\left(x^{2}y^{2}+y^{2}z^{2}+z^{2}x^{2}\right)-\left(1-x^{2}-y^{2}-z^{2}\right)^{3}-b.
\end{aligned}
\]
In the experiment, we take
$
    a=30,
    b=40.
$

We compute the surface integral
\[
    I
    =
    \int_{\Gamma_{a,b}} f(\boldsymbol{x})\,dS,\quad  f(\boldsymbol{x})
    =
    \boldsymbol{x}\cdot\boldsymbol{n}
    =
    \frac{
    \boldsymbol{x}\cdot\nabla\phi_{a,b}(\boldsymbol{x})
    }{
    \lVert\nabla\phi_{a,b}(\boldsymbol{x})\rVert
    }.
\]

For \(a=30\) and \(b=40\), the resulting reference value is
$I_{\mathrm{ref}}
    \approx
    53.6749414237373.$

\begin{figure}[!htbp]
  \centering
  
    \centering
    \includegraphics[width=0.5\linewidth]{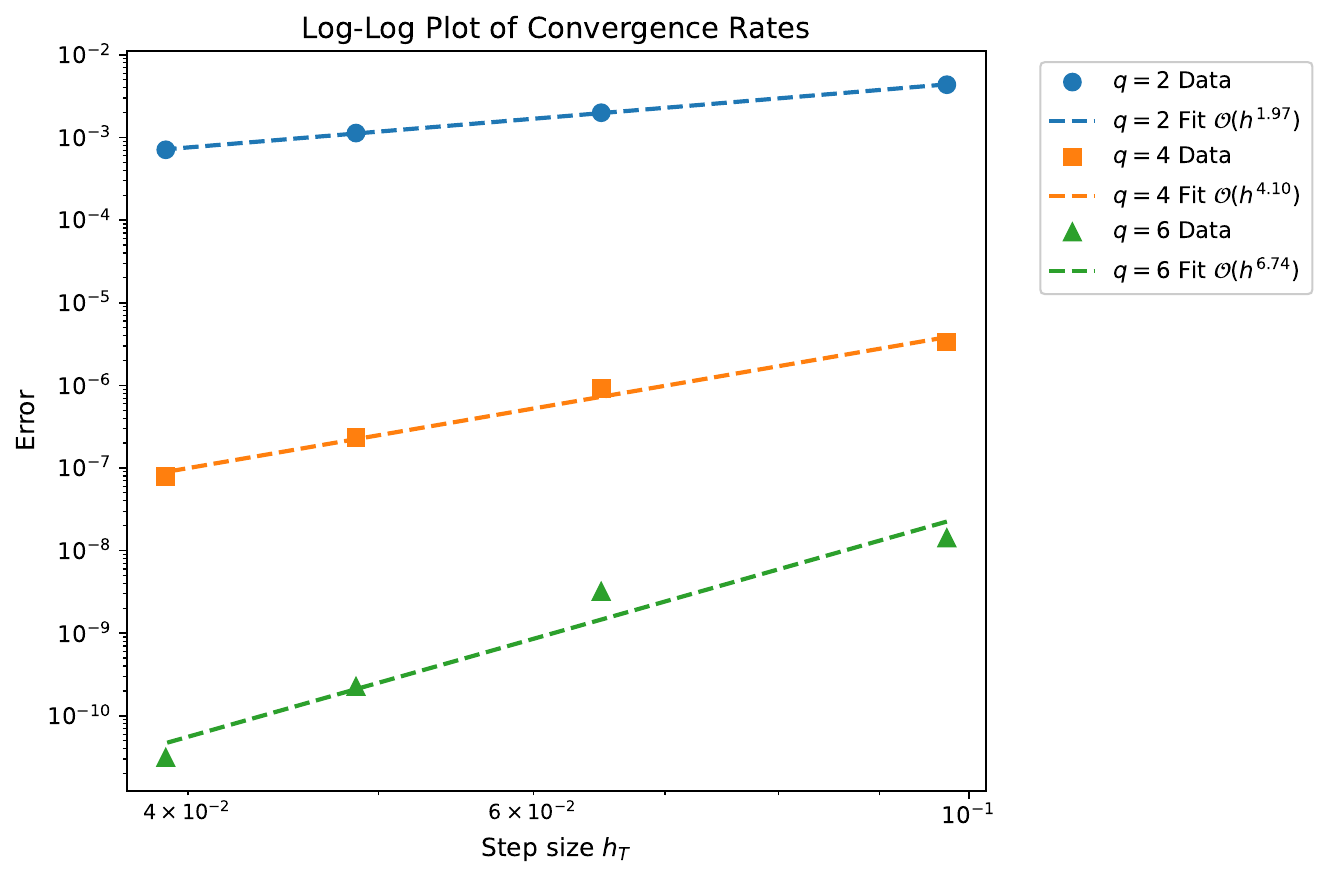} 
  
 \caption{Error in Example 5.\label{fig:Enzensburger} }
\end{figure}
The computational domain is
$[-2.25,2.25]^{3}$.
We use uniform background grids with
   $ =80,120,160,200$ and Gauss-Legendre quadrature
    $q=2,4,6$.
The errors are plotted in Fig.\ref{fig:Enzensburger}.

\section{Conclusion and remarks}
\label{sec:conclusion}

In this paper, we proposed a high-order quadrature method for curves, surfaces, and regions defined implicitly by a smooth level-set function. The method is based on a local change of variables on each cut element. After a mesh adjustment step, the intersection between the background mesh and the zero level set has a simple and nondegenerate structure. This allows the local geometry to be parametrized by solving one-dimensional nonlinear equations along prescribed rays. The resulting quadrature rules are built from standard Gauss--Legendre rules on fixed reference domains.


We established regularity estimates for the local parametrizations and used them to prove high-order convergence for both boundary and region integrals. In particular, for an $N_q$-point Gauss--Legendre rule, the global quadrature error is of order
$
O(h^{N_q})
$
for curve, surface, and region integrals under the stated smoothness and mesh consistency assumptions. Numerical experiments confirm the predicted convergence rates on representative implicit geometries in two and three dimensions.

Several directions remain for future work.  First, The ideas developed in this paper suggest a natural extension to surfaces represented by an unstructured point cloud. For point-cloud data, it would be valuable to combine the present quadrature and mesh-adjustment framework with robust local reconstruction methods and to establish error estimates that account for sampling density, noise, and normal-estimation errors. Second, the extension to adaptive and highly nonuniform meshes requires a more detailed local analysis of the mesh adjustment procedure. Finally, an efficient implementation for large-scale three-dimensional problems is an important direction for practical applications.

\section*{Acknowledgments}
This work was supported by National Natural Science Foundation of China (NSFC) 92370125.

\bibliographystyle{siamplain}

\bibliography{reference}

\end{document}